\documentclass[11pt]{amsart}
\usepackage{amsmath,amsthm,amsfonts,amssymb}
\usepackage{mathrsfs}
\usepackage{enumitem}
\usepackage[utf8]{inputenc}
\usepackage[OT1]{fontenc}
\usepackage{tgtermes}
\usepackage{hyperref}
\hypersetup{  
   bookmarks=true,
   backref=true,
   pagebackref=false,
   colorlinks=true,
   linkcolor=blue,
   citecolor=red,
   urlcolor=blue
}
\usepackage[english]{babel}
\usepackage{amsbsy,amscd,fancyhdr,graphicx,psfrag,fancybox,indentfirst,color}
\usepackage{graphics,epsfig}
\usepackage{fullpage}

\newtheorem{thm}{Theorem}[section]

\newtheorem{lem}[thm]{Lemma}
\newtheorem{prop}[thm]{Proposition}
\newtheorem{defn}[thm]{Definition}
\newtheorem{rem}[thm]{Remark}
\newcommand{\norm}[1]{\|#1\|}
\newcommand{\normL}[1]{\big\|#1\big\|}

\newcommand{\abs}[1]{\left\vert#1\right\vert}
\newcommand{\babs}[1]{\big\vert#1\big\vert}
\newcommand{\pa}[1]{\left(#1\right)}
\newcommand{\bpa}[1]{\big(#1\big)}
\newcommand{\Bpa}[1]{\Big(#1\Big)}
\newcommand{\dotp}[2]{\langle #1,\,#2 \rangle}
\newcommand{\set}[1]{\left\{#1\right\}}

\newcommand{\R}{\mathbb R}
\newcommand{\E}{\mathbb E}
\newcommand{\T}{\mathbb T}
\newcommand{\N}{\mathbb N}
\newcommand{\Z}{\mathbb Z}
\newcommand{\supp}{\mathrm{supp}}
\newcommand{\Sg}{\supp(g)}

\newcommand{\En}{\mathcal{E}_n}
\newcommand{\An}{\mathcal{A}_n}
\newcommand{\LIP}[1]{\mathrm{Lip}(#1)}
\newcommand{\Lip}[1]{L_{#1}}
\newcommand{\diam}{\mathrm{diam}}
\newcommand{\ddt}{\displaystyle\frac{\partial}{\partial t}}
\newcommand{\NeigGama}{\mathcal{N}_{\Gamma}^{a_0}}
\newcommand{\NeigtGameta}{\mathcal{N}_{\tilde{\Gamma}}^{\eta}}
\newcommand{\proj}{\mathrm{Proj}}
\newcommand{\ball}{B}
\newcommand{\distH}{d_{\mathrm{H}}}
\newcommand{\dvol}{\mathrm{d}\vol}
\newcommand{\vol}{\mathrm{vol}}
\newcommand{\qandq}{\quad \text{and} \quad}
\newcommand{\qwithq}{\quad \text{with} \quad}

\def\1{\mathbb I}
\def\a{\alpha}

\def\e{\varepsilon}
\def\g{\gamma}

 \def\O{\Omega}

\makeatletter
\newcommand*\rel@kern[1]{\kern#1\dimexpr\macc@kerna}
\newcommand*\widebar[1]{%
  \begingroup
  \def\mathaccent##1##2{%
    \rel@kern{0.8}%
    \overline{\rel@kern{-0.8}\macc@nucleus\rel@kern{0.2}}%
    \rel@kern{-0.2}%
  }%
  \macc@depth\@ne
  \let\math@bgroup\@empty \let\math@egroup\macc@set@skewchar
  \mathsurround\z@ \frozen@everymath{\mathgroup\macc@group\relax}%
  \macc@set@skewchar\relax
  \let\mathaccentV\macc@nested@a
  \macc@nested@a\relax111{#1}%
  \endgroup
}
\makeatother

\newcommand{\eqdef}{\stackrel{\mathrm{def}}{=}}

\newcommand{\tcb}[1]{{\color{black}{~#1}}}

\begin{document}
\title[]{Limits and consistency of non-local and graph approximations to the Eikonal equation}

\author{Jalal Fadili \and Nicolas Forcadel \and Thi Tuyen Nguyen\and Rita Zantout}

\address{{\sc Jalal Fadili}: Normandie Univ, ENSICAEN, CNRS, GREYC, France.}
\email{Jalal.Fadili@greyc.ensicaen.fr}

\address{{\sc Nicolas Forcadel}: Normandie Univ, INSA de Rouen, Laboratoire de Math\'ematique de l'INSA, Rouen, France.}
\email{nicolas.forcadel@insa-rouen.fr}

\address{{\sc Thi Tuyen Nguyen}: Normandie Univ, INSA de Rouen, Laboratoire de Math\'ematique de l'INSA, Rouen, France.}
\email{thi-tuyen.nguyen@insa-rouen.fr}

\address{{\sc Rita Zantout}: Normandie Univ, INSA de Rouen, Laboratoire de Math\'ematique de l'INSA, Rouen, France.}
\email{rita.zantout@insa-rouen.fr}

\date{\today}
\begin{abstract}
In this paper, we study a non-local approximation of the time-dependent (local) Eikonal equation with Dirichlet-type boundary conditions, where the kernel in the non-local problem is properly scaled. Based on the theory of viscosity solutions, we prove existence and uniqueness of the viscosity solutions of both the local and non-local problems, as well as regularity properties of these solutions in time and space.  We then derive error bounds between the solution to the non-local problem and that of the local one, both in continuous-time and Forward Euler time discretization. We then turn to studying continuum limits of non-local problems defined on random weighted graphs with $n$ vertices. In particular, we establish that if the kernel scale parameter decreases at an appropriate rate as $n$ grows, then almost surely, the solution of the problem on graphs converges uniformly to the viscosity solution of the local problem as the time step vanishes and the number vertices $n$ grows large.
\end{abstract}


\maketitle


\section{Introduction}

In recent years, nonlinear partial differential equations (PDEs) on graphs and networks have attracted increasing interest since they naturally arise in many practical problems in mathematics, physics, biology, economy and data science (e.g., internet and vehicular traffic, social networks, population dynamics, image processing and computer vision, machine learning); see \cite{Berkolaiko13,elb08,Garavello06,Pokornyi04} and references therein. Among those PDEs, Hamilton-Jacobi equations, including Eikonal-type equations, have been considered in \cite{Desquesnes17,del13,Oberman15,Ta09,Ta11,Toutain16} on weighted graphs for data processing, and in \cite{Achdou13,Camilli11,Camilli13,Imbert13,Schieborn13} on topological networks or other very special types of networks. From a different motivation, Hamilton-Jacobi equations on graphs were also studied in \cite{Shu18} to derive discrete versions of some functional inequalities.

Our main goal in this paper is to rigorously study continuum limits of the Eikonal equation defined on weighted graphs, as the number of vertices goes to infinity. The motivation behind considering the Eikonal equation on graphs is the ability to extend it to any discrete data that can be represented by weighted graphs. In such a setting, data points are vertices of the graph, and are connected by edges if sufficiently close in a certain ground metric. The edges are assigned weights (e.g., based on the distances between data points). Several works have considered the approximation of the Eikonal equation on triangular, unstructured meshes or grids; see \cite{Barth98,Carlini13,kimmel98} and references therein. Adaptation of the Eikonal equation on graphs for discrete data processing has been proposed in \cite{Ta09,del13}; see \eqref{eq:eikgraph}. This has led to many applications including semi-supervised clustering and classification on meshes, point clouds, or images~\cite{Ta09,del13,Toutain16,Desquesnes17}; see also \cite{Memoli05} which proposed a framework dedicated to solve the Eikonal equation on point clouds. Despite availability of compelling numerical evidence for the efficiency of \eqref{eq:eikgraph} for these tasks, a clear theoretical understanding of its solutions is lacking and no results on its consistency are available to the best of our knowledge. In particular it is largely open to determine whether solutions of the graph-based Eikonal equation converge, as the number of available data points/vertices increases, to a solution of a limiting equation in the continuum setting. It is our aim in this paper to settle this question.

\subsection{Problem statement}
Here and in the rest of the paper we use $|\cdot|$ to denote the euclidean norm
in $\R^m$,  $\LIP{\Sigma}$ is the space of Lipschitz continuous mappings on $\Sigma$, and for any $h \in \LIP{\Sigma}$, $\Lip{h}$ denotes its Lipschitz constant. For a non-empty closed subset $X \in \R^m$, the distance to $X$ is the function
\[
d(\cdot,X): x \in \R^m \mapsto \min_{z \in X} \abs{x-z} \in [0,+\infty[ .
\]
See also Section~\ref{subsec:notations} for the rest of notations. \\
Throughout, we will work with the following sets and functions satisfying the standing assumptions:\\
\fbox{\parbox{0.975\textwidth}{
\begin{enumerate}[label=({\textbf{H.\arabic*}})]  
\item $\Omega, \tilde{\Omega}$ are compact subsets of $\R^m$, with $\tilde{\Omega}$ a finite subset of $\Omega$; \label{assum:Om}
\item $\Gamma \subset \Omega$ and $\tilde{\Gamma} \subset \tilde{\Omega}$ are closed sets with $\Omega \setminus \Gamma$ open and $\tilde{\Omega} \setminus \tilde{\Gamma} \subset \Omega \setminus \Gamma$; \label{assum:Gam}
\item $P \in \LIP{\Omega \setminus \Gamma}$ and $\tilde{P}\in \LIP{\tilde{\Omega} \setminus \tilde{\Gamma}}$ are non-negative potential functions; \label{assum:P}
\item $\psi \in \LIP{\Omega}$ and $\tilde{\psi} \in \LIP{\tilde{\Omega}}$.  \label{assum:psi}
\item There exists $a_0, d_0>0$ such that $d(\cdot,\Gamma)$ is $C^1$ on $\NeigGama \setminus \Gamma$ where $\NeigGama \eqdef \set{x \in \Omega,\; d(x,\Gamma) < a_0}$,  and $|\nabla d(x,\Gamma)|\ge d_0$ for all $x \in \NeigGama \setminus \Gamma$.\label{assum:regulariteOmega}
\end{enumerate}
}}\\

{Assumptions \ref{assum:Om}-\ref{assum:Gam} imply that $\partial \Omega \subseteq \Gamma$. The fact that $\tilde{\Omega}$ is finite in assumption \ref{assum:Om} is only required to get the existence of a solution for the non-local problem \eqref{cauchy-J} (see Proposition \ref{pro:existence-J}). The rest of our results, for instance our error bounds, do not need this finiteness assumption. Nevertheless, since our primary motivation in this paper is the study of the Eikonal equation on graphs, this assumption is quite natural. Let us also point out that the existence in the case of a continuous set $\tilde{\Omega}$ should be obtained by a discretization of this set.
The regularity assumption \ref{assum:regulariteOmega} on the distance function is classically ensured when, for instance, $\Gamma$ is a compact smooth embedded manifold without boundary, in which case $d_0=1$; see Appendix~\ref{sec:distsmooth} for a discussion. Thus, in the light of the previous observation that $\partial \Omega \subseteq \Gamma$, \ref{assum:regulariteOmega} holds for instance if $\Omega$ is a smooth manifold and $\Gamma$ is a disjoint union of $\partial \Omega$ and a smooth embedded manifold without boundary. We also remark that readers familiar with the theory of viscosity solutions of Hamilton-Jacobi equations may have recognized that \ref{assum:regulariteOmega} is indeed very useful to construct super-solutions that are compatible with boundary conditions.
}


Let $G = (V,w)$ be a finite undirected weighted graph without parallel edges, where \tcb{$V \subset \Omega$} is the vertex set  and every edge $(u,v) \in V^2$ is given a weight $w(u,v)$, where $w: V^2 \to \R_+$ is the weight function. It is understood that $w(u,v)=0$ whenever $(u,v)$ are not connected. In \cite{del13}, the authors proposed the following Eikonal equation on a weighted graph $G$
{
\begin{equation}\label{eq:eikgraph}
\begin{cases}
\max_{v \in V}\sqrt{w(u,v)}(f(v)-f(u))_- = \tilde{P}(u), & u \in V \setminus V_0,\\
f(u) = 0 & u \in V_0,
\end{cases}
\end{equation}
where $(\cdot)_- \eqdef -\min(\cdot,0)$}, and $V_0 \subset V$ corresponds to the set of initial seed vertices. By analogy to the continuous setting, \eqref{eq:eikgraph} describes the evolution of a "propagation front" $V_0$ on the graph $G$. 

In this paper, inspired by \eqref{eq:eikgraph}, we propose to study the general {\textit{non-local}} Eikonal equation in a time-dependent form:
\begin{equation}\tag{\textrm{$\mathcal{P}_{\e}$}}\label{cauchy-J}
\begin{cases}
\ddt f^\e(u,t) = - \abs{\nabla_{J_\e}^- f^\e(u,t)}_\infty + \tilde{P}(u), & (u,t) \in  (\tilde{\Omega}\setminus \tilde{\Gamma}) \times ]0,T[, \\
f^\e(u,t) = \tilde{\psi} (u), & (u,t) \in  (\tilde{\Gamma} \times ]0,T[) \cup \tilde{\Omega} \times \set{0} . 
\end{cases}
\end{equation}
The stationary solution of \eqref{cauchy-J} would satisfy a corresponding Eikonal equation. There are many applications motivating \eqref{eq:eikgraph}, for instance in data processing, analysis and machine learning on graphs, on unstructured meshes or grids, and on point clounds; see \cite{Barth98,Carlini13,Desquesnes17,del13,kimmel98,Oberman15,Memoli05,Ta09,Ta11,Toutain16} with different choices made for the potential $\tilde{P}$, boundary points and data $(\tilde{\Gamma},\tilde{\psi})$.

In \eqref{cauchy-J}, we have defined
\[
\abs{\nabla_{J_\e}^- f^\e(u,t)}_\infty = \max_{v\in\tilde{\Omega}}\nabla_{J_\e}^- f^\e(u,v,t),
\]
where $\nabla_{J_\e}^-$ is a non-local operator coined the weighted directional internal gradient operator, introduced in \cite{del13}, and reads
\[
\nabla_{J_\e}^- f^\e(u,v,t) = J_\e(u,v) (f^\e(v,t) - f^\e(u,t))_-,
\]
where, for $\e > 0$, $J_\e: \R^m \times \R^m \to \R_+$ is an $\e$-scaled kernel function,
\[
J_\e(u,v) = \frac{1}{\e}J\pa{\frac{u}{\e},\frac{v}{\e}} \qwithq  J(u,v) = \frac{1}{C_g}g(|u-v|),
\]
where $C_g > 0$ will be defined in Remark \ref{assum:gubnd}. The above says that the kernel $J$ is isotropic and $g: \R_+ \to \R_+$ is its radial profile. It is easy to see that $\abs{\nabla_{J_\e}^- f^\e(u,t)}_\infty $ can be equivalently rewritten as 
\begin{equation}\label{eq:normnablaJe}
\abs{\nabla_{J_\e}^- f^\e(u,t)}_\infty = \max_{v\in\tilde{\Omega}}J_\e(u,v)(f^\e(u,t)-f^\e(v,t)) .
\end{equation}
In the above, $\e$ is a length scale parameter allowing to take into account data density. Indeed, scaling $J$ by $\e$ is intended to give significant weight to pairs of points up to distance $\e$. To capture properly interactions at scale $\e$, $g$ has to decay to zero at an appropriate rate. More precisely, our set of admissible kernels will have to satisfy the following requirements:\\
\fbox{\parbox{0.975\textwidth}{\begin{enumerate}[label=({\textbf{H.\arabic*}}),start=6] 
\item $g$ is a non-negative function. \label{assum:gpos}
\item $\exists r_g > 0$ such  that $\supp(g) \subset [0,r_g]$. \label{assum:gsupp}
\item $\exists a \in ]0,r_g[$ such that $g$ is decreasing on $[0,a]$ and satisfies $g(a)>0$. We denote by $c_g \eqdef g(a)$. \label{assum:gdec}
\item $g$ is $L_g$-Lipschitz continuous on its support. \label{assum:glip}
\end{enumerate}}}\\
These assumptions on the kernel are mild and rather standard.
\begin{rem} \label{assum:gubnd}
We define 
\[
C_g \eqdef \sup_{t \in \R_+} tg(t),
\]
$C_g$ is definitely bounded by \ref{assum:gsupp} and \ref{assum:glip}. Moreover, in our proof, it turns out that the support compactness assumption \ref{assum:gsupp} is not mandatory. In fact, what is really important is that the supremum in $C_g$ is attained in $[0, r_g]$ and that $g$ is bounded. Nevertheless, to make the proofs simpler to follow, we avoid delving into these technicalities and impose \ref{assum:gsupp}. 
\end{rem}




\eqref{cauchy-J} covers the case of weighted graphs with $n$ vertices as a special case by properly instantiating the sets $(\tilde{\Omega},\tilde{\Gamma})$; see Section~\ref{sec:eikconvgraphs}. Having this in mind, for a given $n$-dependent scaling $\e_n$, we will eventually pose the question of consistency on graphs as the continuum limit of the solution to \eqref{cauchy-J} as $n \to +\infty$, as well as its time discretized versions. We will therefore consider the time-dependent form of the {\textit{local}} Eikonal equation on the continuum:
\begin{equation}\tag{\textrm{$\mathcal{P}$}}\label{cauchy}
\begin{cases}
\ddt  f(x,t) = - \abs{\nabla f(x,t)} + P(x), & (x,t) \in (\Omega \setminus \Gamma) \times ]0,T[,\\
f(x,t) = \psi (x),  & (x,t) \in (\Gamma \times ]0,T[) \cup  \Omega  \times \{ 0 \}  
\end{cases}
\end{equation}
where $\nabla f(x,t)$ denotes the (weak) gradient of $f$ in the space variable $x$.

Before going further, let us perform an easy yet informative calculation just to convince the reader that it is reasonable to hope for a convergence result of a solution of \eqref{cauchy-J} to that of \eqref{cauchy}. More precisely, let us look at the behaviour of the non-local directional internal gradient operator as $\e$ is sent to $0$. For simplicity, we assume that $\tilde{\Omega}=\Omega$  and $\Gamma=\partial \Omega$. Though this case does not comply with assumption \ref{assum:Om}, it gives a fair idea of the computations we will have to carry out rigorously, and we will explain in Section~\ref{subsec:eikconvcont} how to handle properly the case of a discrete $\tilde \Omega$. Let $u \in \Omega \setminus \partial \Omega$. To avoid trivialities, we assume that $\exists v \in \Omega$ such that $|u-v| \in \e \supp(g)$ (this assumption will be discussed in detail later, see Sections~\ref{subsec:eikconvcont} and \ref{sec:eikconvgraphs}). 
If $f$ is differentiable at $u$, then we have for $\e$ sufficiently small,
\begin{align*}
\abs{\nabla_{J_\e}^- f(u,t)}_\infty 
&= \max_{v\in{\Omega}, |u-v| \in \e \supp(g)} J_\e(u,v) (f(u,t) - f(v,t)) \\
&= \max_{v\in{\Omega}, |u-v| \in \e \supp(g)} (\e C_g)^{-1}g\pa{\frac{|u-v|}{\e}} (f(u,t) - f(v,t)) \\
&= \max_{v\in{\Omega}, |u-v| \in \e \supp(g)} (\e C_g)^{-1}g\pa{\frac{|u-v|}{\e}} \bpa{\dotp{\nabla f(u,t)}{u-v} + o(1)} \\
&= \max_{\tau \in [0,r_g]}(\e C_g)^{-1}g(\tau)\max_{v\in{\Omega},|u-v| = \e\tau} \dotp{\nabla f(u,t)}{u-v} + o(1) .
\end{align*}
For $\e$ small enough, we have $\ball_{\e r_g}(u) \subset \Omega$. This entails that
\begin{align*}
\abs{\nabla_{J_\e}^- f(u,t)}_\infty 
&= \max_{\tau \in [0,r_g]}(\e C_g)^{-1}g(\tau)\max_{v \in \ball_{\e\tau}(u)} \dotp{\nabla f(u,t)}{v-u} + o(1) \\
&= \max_{\tau \in [0,r_g]} (\e C_g)^{-1}g(\tau) \e\tau \abs{\nabla f(u,t)} + o(1) \\
&= \abs{\nabla f(u,t)}+ o(1) .
\end{align*}
It is our aim in this paper to give this formal calculation a rigorous meaning and to derive convergence rates. In particular, an important issue in the the previous computation is that the choice of $\e$ depends on the point $u$. We will explain in Section~\ref{subsec:eikconvcont} how to resolve this difficulty.

\subsection{Contributions and relation to prior work}
In this work we intend to provide two related contributions. Their combination allow to quantitatively analyze the Eikonal equation on graph sequences and their continuum limiting behaviour. Our work relies on the important theory of viscosity solutions~\cite{barles11}.

We start by showing that both the local problem \eqref{cauchy} and the non-local one \eqref{cauchy-J} are well-posed, i.e., existence and uniqueness of their viscosity solutions (see Proposition~\ref{pro:existence} and Proposition~\ref{pro:existence-J}). We then establish the regularity properties of these solutions in time and space in Theorem~\ref{lip-viscosity} and Theorem~\ref{lip-viscosity-J}. Capitalizing on this, our first consistency result provides error bounds between the viscosity solutions of \eqref{cauchy-J} and \eqref{cauchy} (Theorem~\ref{thm:continuous-time-estimate}). This is extended to the case where \eqref{cauchy-J} is discretized in time using forward Euler schemes (Theorem~\ref{thm:discrete-fw-estimate}). Though we focus on finite differences in time, due to their popularity and simplicity, we believe that our proof can be adapted to other schemes such as those of semi-Lagrangian type. We finally apply these error bounds to a sequence of random weighted graphs (Theorem~\ref{thm:graph-bw-estimate}). This entails in particular that the time-discretized solution on a weighted graph with $n$ vertices and an appropriately decreasing scale parameter $\e_n$, converges almost surely uniformly to the viscosity solution of \eqref{cauchy} as $n \to +\infty$ and the time step goes to $0$. 

Studying consistency and continuum limits of certain evolution and variational problems on graphs and networks is an active research area; see \cite{Hafiene18,Hafiene20,rossi,medv,medvedevrandom,Trillos16,Slepcev2016,Slepcev2017,Caroccia20} for a non-exhaustive list and references therein. In particular, the authors in \cite{Calder19,Roith21} studied continumm limits of Lipschitz learning on graphs. The Euler-Lagrange equation for Lipschitz learning, as considered in \cite{Calder19}, correspond to a stationary special case of \eqref{cauchy-J}, where the operator \eqref{eq:normnablaJe} is replaced by the $\infty$-Laplacian on graphs, $\tilde{P} \equiv 0$, $\tilde{\Omega}$ is a set of $n$ points in the flat torus $\T^m = \R^m/\Z^m$, and $\tilde{\Gamma} \subset \tilde{\Omega}$ is a fixed finite collection of points. In such a setting, it is proved in \cite{Calder19} that the solution of the discrete problem converges uniformly to the unique viscosity solution of an $\infty$-Laplace type equation on the flat torus, as $\e_n \to 0$ when $n \to +\infty$. The limit equation turns out to be the stationary version of \eqref{cauchy} where $\Omega=\T^m$, $P \equiv 0$, $\psi = \tilde{\psi}$, $\Gamma = \tilde{\Gamma}$, and $\abs{\nabla \cdot}$ is replaced with the $\infty$-Laplacian. While finalizing our paper, we also became aware of the recent work of \cite{Roith21} who used tools from $\Gamma$-convergence theory to prove asymptotic consistency of Lipschitz learning on graphs (though the $\Gamma$-limit is not unique), allowing moreover that $\Omega$ to be a sufficiently smooth closed set and $\tilde{\Gamma}$ possibly different from $\Gamma$. In both \cite{Calder19,Roith21}, for consistency to hold, it is required that the Hausdorff distance between $\tilde{\Omega}$ and $\Omega$ (and also between $\tilde{\Gamma}$ and $\Gamma$) is $o(\e^\alpha)$ ($\alpha=3/2$ in \cite{Calder19} and $\alpha=1$ in \cite{Roith21}). Some of their assumptions are different or stronger from those we require in this paper. In addition, we are not aware of any work which establishes continuum limits for an Eikonal equation on weighted graphs of the form \eqref{cauchy-J}. We also provide and analyze convergence of a concrete forward Euler discrete-time scheme to solve \eqref{cauchy-J}. It is also important to stress the fact that our primary interest is in providing error bounds and non-asymptotic convergence rates. This is known to be more challenging than deriving mere asymptotic consistency results.

Motivated by a continuous version of the shortest path problem, numerical approximations of the Hamilton-Jacobi equations of Eikonal-type defined on a topological network were studied in \cite{Camilli11,Camilli13}. A topological network is basically a graph embedded in Euclidean space, i.e., it is a collection of pairwise different points (vertices) in a Euclidean space connected by differentiable, non self-intersecting curves (smooth edges). This is a very special network structure far different from the weighted graph setting we study here.   

\subsection{Paper organization}
The paper is organized as follows. In Section~\ref{sec:wellposed}, we show that \eqref{cauchy} and \eqref{cauchy-J} are well-posed in the sense of viscosity solutions, and we establish some important regularity results that will be central in our error bounds. Section~\ref{sec:main} states the main results of this paper. We start with a key error bound between solutions of problems \eqref{cauchy-J} and \eqref{cauchy} in both time-continuous case and time-discrete cases using explicit/forward Euler schemes. We then turn to applying these results to weighted graphs in Section~\ref{sec:eikconvgraphs}.

\subsection{Notations}\label{subsec:notations}
In what follows, we will denote $\dotp{\cdot}{\cdot}$ the scalar product on $\R^m$, and $\ball_r(x)$ the Euclidean ball centered at $x \in \R^m$ of radius $r$. For a non-empty and closed subset $X \in \R^m$ and $x \in \R^m$, we denote by $\proj_X(x)$ the projection of $x$ on $X$, i.e., the set of nearest points of $x$ in $X$: 
\[
\proj_X(x) = \set{z \in X: \abs{x-z} = d(x,X)} .
\]
Since $X$ is non-empty and closed, $\proj_X(x)$ is non-empty at any $x \in \R^m$ but is not necessarily single-valued. The diameter of $X$ is $\diam(X) = \sup_{(x,z) \in X^2}|x-z|$. Let $X$ and $Y$ be two non-empty subsets of $\R^m$. Their Hausdorff distance is defined as 
\[
\distH(X,Y) = \max\pa{\sup_{x \in X} d(x,Y),\sup_{y \in Y} d(y,X)} .
\]
It is finite when $X$ and $Y$ are bounded, and when $X$ and $Y$ are closed, then $\distH(X,Y)=0 \iff X=Y$.

We will denote $\normL{\cdot}_{L^\infty(\Sigma)}$, the supremum norm on a domain $\Sigma \subset \R^m$. 

To lighten notation, we denote the bounded space-time cylinders $\Omega_T = \Omega \times [0,T]$ and $\partial \Omega_T = (\Gamma \times ]0,T[) \cup \Omega \times \{ 0\}$.
We define similarly $\tilde{\Omega}_{T}$ and $\partial \tilde{\Omega}_{T}$. For a time interval $[0,T[$ and $N_T \in \N$, we also use the shorthand notation $\tilde{\Omega}_{N_T} = \tilde{\Omega} \times \{0,\ldots,t_{N_T} \}$ and $\partial \tilde{\Omega}_{N_T} = \tilde{\Gamma} \times \{t_1,\ldots,t_{N_T} \} \cup  \tilde{\Omega}  \times \{ 0 \}$.

\section{Well-posedness and regularity results}
\label{sec:wellposed}

\subsection{Problem~\eqref{cauchy}}
Since we will work with viscosity solutions, we refer to \cite{barles11,cil92,cl86,i86} for a good introduction. 
\tcb{The notion of viscosity solutions was introduced by Crandall and Lions \cite{cl83} as a notion of weak solutions for a class of partial differential equations of Hamilton-Jacobi type. In this theory, the derivatives of the unknown  are replaced by the derivative of some test functions (see definition below).}
In order to give the definition of  viscosity solution for problem \eqref{cauchy}, we first recall the definition of upper and lower semi-continuous envelope for a locally bounded function $f:\Omega_T \to \R$, respectively given by
\[
f^*(x,t) \eqdef \limsup_{(y,s)\to (x,t)} f(y,s) \qandq 
f_*(x,t) \eqdef \liminf_{(y,s)\to (x,t)} f(y,s).
\]

\begin{defn}[Viscosity solution for \eqref{cauchy}]\label{def:visccauchy}
An upper semi-continuous function (usc) function $f:\Omega_T \to \R$ is a viscosity sub-solution of \eqref{cauchy}
in $(\Omega \setminus \Gamma) \times ]0,T[$  if  for any $\varphi \in C^1((\Omega \setminus \Gamma ) \times ]0,T[)$ such that $f-\varphi$ reaches a local maximum point at $(x_0,t_0) \in (\Omega \setminus \Gamma) \times ]0,T[$, one has
\begin{eqnarray*}
\ddt \varphi(x_0,t_0) \leq - \vert \nabla \varphi(x_0,t_0) \vert + P(x_0).
\end{eqnarray*}
The function $f$ is a viscosity sub-solution of \eqref{cauchy} in $\Omega_T$ if it satisfies moreover $f(x,t) \le \psi(x)$ for all $(x,t) \in \partial \Omega_T$.

A lower semi-continuous (lsc) function $f:\Omega_T \to \R$ is a viscosity super-solution of \eqref{cauchy} in $(\Omega \setminus \Gamma) \times ]0,T[$ if for any $\varphi \in C^1((\Omega \setminus \Gamma ) \times ]0,T[)$ such that $f-\varphi$ attains a local minimum point at $(x_0,t_0) \in (\Omega \setminus \Gamma) \times ]0,T[$, one has
\begin{eqnarray*}
\ddt \varphi(x_0,t_0) \geq - \vert \nabla \varphi(x_0,t_0) \vert + P(x_0).
\end{eqnarray*}
The function $f$ is a viscosity super-solution of \eqref{cauchy} in $\Omega_T$ if it satisfies moreover $f(x,t)\ge \psi(x)$ for all $(x,t) \in \partial \Omega_T$. %
 
Finally, a locally bounded function $f :\Omega_T\to \R$ is a viscosity solution of \eqref{cauchy} in $(\Omega \setminus \Gamma) \times ]0,T[$ (resp. in $\Omega_T$ ) if  $f^*$ is a viscosity sub-solution and $f_*$ is a viscosity super-solution of \eqref{cauchy} in $(\Omega \setminus \Gamma) \times ]0,T[$ (resp. in $\Omega_T$).
\end{defn}


We continue with a comparison principle for problem \eqref{cauchy}.
\begin{prop}[Comparison principle for~\eqref{cauchy}]\label{pro:PC}
Suppose that assumptions~\ref{assum:Om}--\ref{assum:psi} hold. Let $f$, an usc function, be a sub-solution of \eqref{cauchy} and $g$, a lsc function, be a super-solution of \eqref{cauchy}. Then 
\[
f \le g \quad {\rm in}\;  \Omega_T.
\]
\end{prop}
\begin{proof}
The proof can be found in \cite[Theorem~5.1, Remark~5.1]{barles11}.
\end{proof}

It is well-known that \eqref{cauchy}, which accounts for a "Dirichlet-type" boundary condition, cannot be solved for any function $\psi$; see e.g., \cite[Section~2.6.3]{BarlesBook}. Thus, in order to construct solutions to \eqref{cauchy}, compatibility properties between the equation and the boundary conditions are necessary. This is precisely what we impose through the following assumption:\\
\fbox{\parbox{0.975\textwidth}{
\begin{enumerate}[label=({\textbf{H.\arabic*}}),start=10]
\item There exists $\psi_b\in \LIP{\Omega}$, with $\psi_b(x)=\psi(x)$ for all $x \in \Gamma$, such that $\psi_b$ is a sub-solution of \eqref{cauchy} in $\Omega_T$. \label{assum:psisssol}
\end{enumerate}}}\\
\begin{rem}\label{rem:asspsi}
To give better intuition and insight of \ref{assum:psisssol}, let's give relevant examples of problems that satisfy this assumption. First consider the simplified model, posed in $\R$ and without dependence in time:
\[
\begin{cases}
|u'|=1	&{\rm in}\; ]-1,1[\\
u(x)=Kx	&{\rm for}\; x \in \{-1,1\} ,
\end{cases}
\]
for some constant $K$. Clearly, if $|K|>1$, it is not possible to construct a solution satisfying the boundary condition in a strong sense. However, it is still possible to construct a solution satisfying the boundary condition in a weak sense (i.e. either the boundary condition or the equation is satisfied at the boundary; see \cite{BarlesBook}) but this is not what we want to do here. This is the reason why we have to impose a kind of compatibility condition between the boundary conditions and the equation. For this example, assumption \ref{assum:psisssol} is ensured as soon as $|K|\le 1$.

\medskip

Another prominent example where assumption \ref{assum:psisssol} is satisfied is when $\psi=0$ and $P\ge0$ in \eqref{cauchy}. This setting corresponds to a time-dependent Eikonal equation where the steady state solution can be interpreted as the shortest traveling time or distance of a point $x \in \Omega$ to front $\Gamma$, where the travel inverse speed is $P$. This example plays an important role for computing distance functions which is a key step in numerous applications including image processing or computational computational geometry \cite{Sethian88,Sethian96,Sethian99}.
\end{rem}


\begin{rem}\label{rem:psib}
Assumption \ref{assum:psisssol} entails in particular that 
\[
|\nabla \psi_b(x)|\le \norm{P}_{L^\infty(\Omega\setminus\Gamma)} ,
\]
and thus the Lipschitz constant $L_{\psi_b}$ satisfies 
\[
L_{\psi_b}\le \norm{P}_{L^\infty(\Omega\setminus\Gamma)} .
\]
\end{rem}

We then have the following result which gives the existence and uniqueness of viscosity solution for problem \eqref{cauchy}.
\begin{prop}[Existence and uniqueness for \eqref{cauchy}]\label{pro:existence}
Suppose that assumptions~\ref{assum:Om}--\ref{assum:regulariteOmega} and \ref{assum:psisssol} hold.
\tcb{Then, problem \eqref{cauchy} admits a unique viscosity solution $f$ (which is in fact continuous).} Moreover, there exists a function $\widebar{f}\in \LIP{\Omega_T}$, with a Lipschitz constant depending on $a_0,\; d_0,\; \Lip{\psi}$ and $\norm{P}_{L^\infty(\Omega\setminus\Gamma)}$,  such that
\begin{equation}\label{eq:barrier-f}
\psi_b \leq f \leq \widebar{f} \quad {\rm in }\;\Omega_T.
\end{equation}
\end{prop}
Before giving the proof of this proposition, we first define the notion of barrier solutions and then recall Perron's method.

\begin{defn}[Barrier sub- and super-solution]\label{def:barriercauchy}
An usc function $\underline f :\Omega_T\to \R$ is a barrier sub-solution of \eqref{cauchy} if it is a viscosity sub-solution
in $(\Omega \setminus \Gamma) \times ]0,T[$ and if it satisfies moreover
\[
\lim_{y\to x, s\to t} \underline f(y,s)=\psi(x)\quad \forall (x,t)\in \Gamma \times [0,T].
\]
A lsc function $\widebar{f} :\Omega_T\to \R$ is a barrier super-solution of \eqref{cauchy} if it is a viscosity super-solution in $(\Omega \setminus \Gamma) \times ]0,T[$ and if it satisfies moreover
\[
\lim_{y\to x, s\to t} \widebar{f}(y,s)= \psi(x)\quad \forall (x,t) \in \Gamma \times [0,T].
\]
\end{defn}

\begin{thm}[Perron's method~\cite{Ishii87}]\label{perron}
Assume that there exists a barrier sub-solution $\underline f$ and a barrier super-solution $\widebar{f}$ of \eqref{cauchy}. Then there exists a (possibly discontinuous) viscosity solution $f$ of \eqref{cauchy} satisfying moreover
\[
\underline f\le f\le \widebar{f}\quad {\rm in}\; \Omega_T.
\]
\end{thm}

We are now ready to prove Proposition~\ref{pro:existence}.
\begin{proof}[Proof of Proposition~\ref{pro:existence}]
By assumption \ref{assum:psisssol}, $\psi_b$ is a barrier sub-solution of \eqref{cauchy}. We then have to construct a barrier super-solution $\widebar{f}$. Existence will then be a direct consequence of Perron's method as recalled in Theorem~\ref{perron} while uniqueness and continuity will be direct consequences of the comparison principle shown in Proposition~\ref{pro:PC}.
%

Let 
\[
\widebar{f}_1(x,t)= \psi(x)+K_1t, \quad (x,t) \in \Omega_T ,
\]
where $K_1=\norm{P}_{L^\infty(\Omega\setminus\Gamma)}$, and 
\[
\widebar{f}_2(x,t)= \psi(x)+K_2 d(x,\Gamma), \quad (x,t) \in \Omega_T ,
\] 
with $K_2 > 0$ large enough to be determined later. We set 
\begin{equation}\label{eq:fmindef}
\widebar{f}(x,t) = \min(\widebar{f}_1(x,t),\widebar{f}_2(x,t))= \min(\psi(x)+K_1t,  \psi(x)+K_2d(x,\Gamma)).
\end{equation}
We claim that $\widebar{f}$ is a barrier super-solution.

First, observe that 
\[
\Lip{\widebar{f}} \leq \max\bpa{\Lip{\widebar{f}_1},\Lip{\widebar{f}_2}} \leq \Lip{\psi}+\max\pa{\norm{P}_{L^\infty(\Omega\setminus\Gamma)},K_2},
\]
since $\psi \in \LIP{\Omega}$ by \ref{assum:psi} and $\Lip{d(\cdot,\Gamma)}=1$ as $\Gamma \neq \emptyset$. In particular, $\widebar{f}$ is continuous. 

Moreover, we have for $x \in \Gamma$, 
\[
\widebar{f}_2(x,t)= \psi(x) \leq \widebar{f}_1(x,t).
\]
Hence 
\begin{equation}\label{eq:fminGamma}
\widebar{f}(x,t) =  \psi(x), \quad \forall (x,t) \in \Gamma \times [0,T],
\end{equation}
which shows, via continuity that the limit property required in Definition~\ref{def:barriercauchy} holds. It remains to prove that $\widebar{f}$ is a super-solution on $(\Omega\setminus \Gamma)\times ]0,T[$ for $K_2$ large enough. 

Observe first that by taking $K_2 \geq K_1T/a_0$, we have for all $x \in \Omega \setminus \NeigGama$ (recall that $\NeigGama$ is defined in assumption \ref{assum:regulariteOmega}),
\[
\widebar{f}_2(x,t) \ge  \psi(x) + K_2 a_0 \ge \psi(x) + K_1 T  \ge \widebar{f}_1(x,t) ,
\]
and thus \eqref{eq:fmindef} becomes
\begin{equation}\label{eq:fcases}
\widebar{f}(x,t) = 
\begin{cases}
\min(\widebar{f}_1(x,t),\widebar{f}_2(x,t)) & \text{ if } (x,t) \in \NeigGama \times [0,T], \\
\widebar{f}_1(x,t) & \text{ if } (x,t) \in \Omega \setminus \NeigGama \times [0,T] .
\end{cases}
\end{equation}
Following Definition~\ref{def:visccauchy}, let $\varphi\in C^1((\Omega\setminus \Gamma)\times ]0,T[)$ such that $\widebar{f}-\varphi$ reaches a local minimum at some $(x_0,t_0)\in (\Omega \setminus \Gamma)\times ]0,T[$. This is equivalent to 
\begin{equation}\label{eq:fineqmin}
\widebar{f}(y,s)-\varphi(y,s) \geq \widebar{f}(x_0,t_0)-\varphi(x_0,t_0) ,
\end{equation}
for all $(y,s) \in (\Omega\setminus \Gamma)\times ]0,T[$ sufficiently close to $(x_0,t_0)$. We now distinguish different cases.

\begin{enumerate}[label={\bf Case~\arabic*}]
\item~$x_0 \in \Omega \setminus \NeigGama$. \label{case:f1}
In this case, since $(\Omega \setminus \NeigGama) \subset (\Omega \setminus \Gamma)$, it follows from \eqref{eq:fcases} and \eqref{eq:fineqmin} that
\[
\widebar{f}_1(y,s)-\varphi(y,s) \geq \widebar{f}(y,s)-\varphi(y,s) \geq \widebar{f}_1(x_0,t_0)-\varphi(x_0,t_0) ,
\]
for all $(y,s) \in (\Omega\setminus \Gamma)\times ]0,T[$ sufficiently close to $(x_0,t_0)$. As $]0,T[$ is open, we take $y=x_0$ and $s=t_0+h \in ]0,T[$ for $h>0$ sufficiently small, which gives us
\begin{equation}\label{eq:phif1ineq}
\varphi(x_0,t_0+h)-\varphi(x_0,t_0) \leq \widebar{f}_1(x_0,t_0+ h)-\widebar{f}_1(x_0,t_0)=K_1h.
\end{equation}
Dividing by $h$ and passing to the limit as $h \to 0^+$, we get 
\begin{equation}\label{eq:dphif1ubnd}
\ddt\varphi (x_0,t_0) \leq K_1 .
\end{equation}
Embarking from \eqref{eq:phif1ineq} where we replace $h$ by $-h$ yields 
\begin{equation}\label{eq:dphif1lbnd}
\ddt\varphi(x_0,t_0) \geq K_1, 
\end{equation}
and thus
\begin{equation}\label{eq:dphif1}
\ddt\varphi(x_0,t_0)= K_1 .
\end{equation}
We then deduce that\footnote{Actually, only the lower bound inequality \eqref{eq:dphif1lbnd} is needed here.}
\begin{equation}\label{eq:superresineq1}
\ddt\varphi(x_0,t_0)+|\nabla \varphi(x_0,t_0)|-P(x_0) \ge K_1-P(x_0)\ge K_1 - \norm{P}_{L^\infty(\Omega\setminus\Gamma)} = 0,
\end{equation}
which shows  the desired inequality in this case\footnote{In fact, this portion of the proof shows that $\widebar{f}_1$ is a super-solution of \eqref{cauchy} in $(\Omega \setminus \Gamma)\times ]0,T[$.}.

\item~$x_0 \in \NeigGama \setminus \Gamma$. \label{case:f2}
Let $I_0 \eqdef \set{i \in \set{1,2}:~ \widebar{f}(x_0,t_0) = \widebar{f}_i(x_0,t_0)}$. Thus, for any $i_0 \in I_0$, we have from \eqref{eq:fineqmin} that
\begin{equation}\label{eq:fi0ineqmin}
\widebar{f}_{i_0}(y,s)-\varphi(y,s) \geq \widebar{f}(y,s)-\varphi(y,s) \geq \widebar{f}(x_0,t_0)-\varphi(x_0,t_0) = \widebar{f}_{i_0}(x_0,t_0)-\varphi(x_0,t_0) 
\end{equation}
for all $y \in \NeigGama \setminus \Gamma$ close enough to $x_0$. If $1 \in I_0$ then we are done thanks to \ref{case:f1}. It remains to consider the case where $I_0 = \set{2}$. Embarking from \eqref{eq:fi0ineqmin} with $i_0=2$, and arguing as we have done for $\widebar{f}_1$ in \ref{case:f1} to show \eqref{eq:dphif1}, and using that $\widebar{f}_2$ is actually $t$-independent, we get in this case that 
\begin{equation}\label{eq:dphif2}
\ddt\varphi(x_0,t_0)= 0 .
\end{equation}
On the other hand, since $\NeigGama \setminus \Gamma$ is open by \ref{assum:Gam} and \ref{assum:regulariteOmega}, we have $y=x_0+hz \in \NeigGama \setminus \Gamma$ for $h>0$ small enough and any $z \in \R^m$ such that $|z|=1$. Thus, in view of \ref{assum:regulariteOmega}, inequality \eqref{eq:fi0ineqmin} becomes
\[
\frac{(\varphi(\cdot,t_0)-K_2d(\cdot,\Gamma))(x_0+hz)-(\varphi(\cdot,t_0)-K_2d(\cdot,\Gamma))(x_0)}{h} \ge \frac{\psi(x_0+hz)-\psi(x_0)}{h}\ge -\Lip{\psi}.
\]
Passing to the limit as $h \to 0^+$, we get
\[
\dotp{\nabla \varphi(x_0,t_0)-K_2\nabla d(x_0,\Gamma)}{z} \ge -\Lip{\psi} .
\]
If $\nabla \varphi(x_0,t_0)-K_2\nabla d(x_0,\Gamma) = 0$, we have from \eqref{eq:dphif2} and \ref{assum:regulariteOmega} that
\begin{align}\label{eq:superresineq2a}
\ddt\varphi(x_0,t_0) + |\nabla \varphi (x_0,t_0)|-P(x_0) = K_2\nabla d(x_0,\Gamma) - P(x_0) \geq K_2d_0 - \norm{P}_{L^\infty(\Omega\setminus\Gamma)} \geq 0
\end{align}
for $K_2 \ge \norm{P}_{L^\infty(\Omega\setminus\Gamma)}/d_0$. In the case where $\nabla \varphi(x_0,t_0)-K_2\nabla d(x_0,\Gamma) \neq 0$, we choose
\[
z=\pm \frac{\nabla \varphi(x_0,t_0)-K_2 \nabla d(x_0,\Gamma)}{\abs{\nabla \varphi(x_0,t_0)-K_2 \nabla d(x_0,\Gamma)}}
\]
to arrive at
\[
|\nabla \varphi(x_0,t_0)-K_2\nabla d(x_0,\Gamma)|\le \Lip{\psi}.
\]
Combining this inequality with \eqref{eq:dphif2} and \ref{assum:regulariteOmega}, we get
\begin{align}\label{eq:superresineq2b}
\ddt\varphi(x_0,t_0)+|\nabla \varphi (x_0,t_0)|-P(x_0)\ge&  K_2|\nabla d(x_0,\Gamma)|-|\nabla \varphi(x_0,t_0)-K_2\nabla d(x_0,\Gamma)|-P(x_0) \nonumber\\
\ge&  K_2d_0-\Lip{\psi}-\norm{P}_{L^\infty(\Omega\setminus\Gamma)} \nonumber\\
\ge& 0
\end{align}
for $K_2 \ge( \Lip{\psi}+\norm{P}_{L^\infty(\Omega\setminus\Gamma)})/d_0$.
\end{enumerate}
In summary, taking $K_2 \geq \max\pa{(\Lip{\psi}+\norm{P}_{L^\infty(\Omega\setminus\Gamma)})/d_0,K_1T/a_0}$, the inequalities \eqref{eq:superresineq1}, \eqref{eq:superresineq2a} and \eqref{eq:superresineq2b} hold in each respective case, and thus the desired super-solution inequality is satisfied in all cases. We then conclude that $\widebar{f}$ is a barrier super-solution. The existence of $f$ and the bound \eqref{eq:barrier-f} are then direct consequences of Perron's method.
%
%
\end{proof}

We finish this section by a regularity result.
\begin{thm}[Regularity of the solution of \eqref{cauchy}]\label{lip-viscosity}
Suppose that assumptions~\ref{assum:Om}--\ref{assum:regulariteOmega} and \ref{assum:psisssol} hold. Then the unique viscosity solution to the problem \eqref{cauchy}  satisfies the following regularity properties
\begin{align}
f(x,\cdot) &\in \LIP{[0,T[} &\qwithq \Lip{f(x,\cdot)} &\leq \norm{P}_{L^\infty(\Omega\setminus\Gamma)} + \Lip{\psi}, && \forall x \in \Omega, \label{lip-t-ct-1} \\
f(\cdot,t) &\in \LIP{\Omega} &\qwithq \Lip{f(\cdot,t)} &\leq  2\norm{P}_{L^\infty(\Omega\setminus\Gamma)}+\Lip{\psi}, && \forall t \in [0,T] . \label{eq:lipspace}
\end{align}
\end{thm}

\begin{proof}
When $x \in {\Gamma}$, \eqref{lip-t-ct-1} obviously holds. It remains to consider the case $(x,t) \in ({\Omega} \setminus {\Gamma}) \times [0,T[$. 

Let $h>0$ sufficiently small and set $l(x,t) = f(x,t+h)$  for all $(x,t) \in \Omega_T$. One then has $l(x,0) = f(x,h)$ for all $(x,t) \in (\Gamma \times ]0,T[) \cup  \Omega  \times \{ 0 \}$ and thus it is easy to verify that $l$ satisfies
\begin{equation*}
\begin{cases}
\ddt  l(x,t) = - \vert \nabla l(x,t)\vert + P(x), & (x,t) \in (\Omega \setminus \Gamma) \times ]0,T[,\\
l(x,t) = f(x,h),  & (x,t) \in (\Gamma \times ]0,T[) \cup  \Omega  \times \{ 0 \}  .
\end{cases}
\end{equation*}
This entails that $f(x)$ and $f(x,t+h)$ are solutions of the same equation \eqref{cauchy}, respectively with  initial conditions $\psi$ and $f(x,h)$. Applying again the comparison result of Proposition~\ref{pro:PC} we have
\begin{equation}\label{eq:lipfth}
|f(x,t+h) - f(x,t)| \leq |f(x,h) - \psi(x)| .
\end{equation}

To conclude, it remains to show that the right hand side of \eqref{eq:lipfth} is $O(h)$. Let us define, for  $(x,t) \in \Omega_T$,
\[
f_1(x,t) = {\psi}(x) - Lt; \qquad  f_2(x,t) = {\psi}(x) + Lt ,
\]
where $L = \norm{P}_{L^\infty(\Omega\setminus\Gamma)} + \Lip{\psi}$. 
Arguing in the same way as we have done for $\widebar{f}_1$ in the proof of Proposition~\ref{pro:existence}, we get that $f_1$ and $f_2$ are respectively a sub- and a super-solution of \eqref{cauchy}. Hence, by the comparison principle in Proposition~\ref{pro:PC}, we obtain 
\begin{eqnarray*}
 {\psi}(x) - Lt \leq f(x,t) \leq  {\psi}(x) + Lt  ,
\end{eqnarray*}
whence we get
\begin{eqnarray}\label{compar-ineq-ct-1}
|f(x,t) - \psi(x)| \leq L t ,
\end{eqnarray}
Combining \eqref{eq:lipfth} and \eqref{compar-ineq-ct-1} yields
\begin{equation*}
|f(x,t+h) - f(x,t)| \leq  L h. 
\end{equation*}

\medskip

We now turn to the space regularity bound \eqref{eq:lipspace} and adapt the argument of \cite[Theorem~8.1]{barles11}. We introduce the test-function 
\[
\Psi: (x,t,y) \in \Omega_T \times \Omega \mapsto f(x,t)-f(y,t)-K|x-y| ,
\]
and we aim at showing that $\Psi$ is negative for sufficiently large $K > 0$. When $t=0$ or $(x,y) \in \Gamma^2$, we can choose $K \ge \Lip{\psi}$ to have that \eqref{eq:lipspace} holds.

We argue by contradiction, assuming that
\[
\sup_{(x,t,y) \in \Omega_T \times \Omega} \Psi(x,t,y) > 0 .
\]
Continuity of $f$ and compactness of $\Omega_T \times \Omega$ entail that the supremum of $\Psi$ is actually a maximum attained at some point $(\bar{x},\bar{t},\bar{y}) \in \Omega_T \times \Omega$. In order to use viscosity solutions arguments, we use the classical doubling of the variable in time, and introduce the function, for $\alpha > 0$,
\[
\Psi_\a: (x,t,y,s) \in \Omega_T^2 \mapsto f(x,t)-f(y,s)-K|x-y|-\frac{|t-s|^2}{2\a} .
\]
This function has a maximum attained at some point in $\Omega_T^2$, say $(x_\a,t_\a,y_\a,s_\a)$. We obviously have 
\begin{equation}\label{eq:Psiapos}
\Psi_\a(x_\a,t_\a,y_\a,s_\a) \geq \Psi_\a(\bar{x},\bar{t},\bar{y},\bar{t}) = \Psi(\bar{x},\bar{t},\bar{y}) > 0 .
\end{equation}
Observe also that for $\alpha$ sufficiently small, we cannot have $x_\a = y_\a$ as otherwise $\Psi_\a(x_\a,t_\a,y_\a,s_\a)$ would be negative, hence contradicting \eqref{eq:Psiapos}.

If $x_\a\in \Gamma$, then $f(x_\a,t_\a)=\psi(x_\a)=\psi_b(x_\a)$. Moreover, by  Proposition~\ref{pro:existence}, $\psi_b(y_\a) \leq f(y_\a,s_\a)$. It then follows that
\begin{align*}
\Psi_\a(x_\a,t_\a,y_\a,s_\a) \le & \psi_b(x_\a)-f(y_\a,s_\a)-K|x_\a-y_\a|\\
\le& \psi_b(x_\a)-\psi_b(y_\a)-K|x_\a-y_\a|\\
\le& (L_{\psi_b}-K)|x_\a-y_\a|\\
\le& (\norm{P}_{L^\infty(\Omega\setminus\Gamma)}-K)|x_\a-y_\a| 
\end{align*}
where we used Remark~\ref{rem:psib}. Taking $K \geq \norm{P}_{L^\infty(\Omega\setminus\Gamma)}$ contradicts positivity of $\Psi_\a(x_\a,t_\a,y_\a,s_\a)$ on $\Gamma$. 

Consider in the rest the case $x_\a\in \Omega \setminus \Gamma$. Since $x_\a \neq y_\a$, the function $(x,t) \mapsto f(y_\a,s_\a)+K|x-y_\a|+\frac{|t-s_\a|^2}{2\a}$ is smooth at $(x_\a,t_\a)$ and since $f$ is sub-solution, we have
\[
\frac{t_\a-s_\a}\a + K \le P(x_\a) .
\]
On the other hand, maximality of $\Psi_\a$ at $(x_\a,t_\a,y_\a,s_\a)$ implies, for all $t\in[0,T]$
\[
f(x_\a,t)-\frac{|t-s_\a|^2}{2\a} \leq f(x_\a,t_\a)-\frac{|t_\a-s_\a|^2}{2\a}.
\]
Choosing $t$ such that $t_\a-s_\a$ and $t_\a-t$ are of same sign, and using \eqref{lip-t-ct-1}, we get
\begin{align*}
L|t-t_\a|\ge & f(x_\a,t_\a)-f(x_\a,t)\\
\ge & \frac{|t_\a-s_\a|^2}{2\a}-\frac{|t-s_\a|^2}{2\a}\\
=&-\frac{|t-t_\a|^2}{2\a}+|t-t_\a|\frac{|t_\a-s_\a|}\a.
\end{align*}
Dividing by $|t-t_\a|$ and taking $t\to t_\a$, we get
\[
\frac{|t_\a-s_\a|}{\a} \leq L .
\]
Hence,
\[
K \leq \norm{P}_{L^\infty(\Omega\setminus\Gamma)}+L .
\]
Choosing $K > \norm{P}_{L^\infty(\Omega\setminus\Gamma)}+L$ we get again a contradiction of the positivity of $\Psi_\a(x_\a,t_\a,y_\a,s_\a)$ on $\Omega \setminus \Gamma$. The above proof shows then that 
\[
f(x,t)-f(y,t)-K|x-y| \leq 0
\] 
fo all $(x,y,t) \in \Omega^2 \times [0,T]$ and every $K>2\norm{P}_{L^\infty(\Omega\setminus\Gamma)}+L_{\psi}$, i.e., $f(\cdot,t)$ is globally Lipschitz continuous uniformly in $t$, hence providing the bound \eqref{eq:lipspace}.
%
\end{proof}

%

\subsection{Problem~\eqref{cauchy-J}}

We begin by the definition of viscosity solution for problem \eqref{cauchy-J}.
\begin{defn}[Viscosity solution for \eqref{cauchy-J}]
An usc function $f^\e : \tilde{\Omega}_T \to \R$ is a viscosity sub-solution of \eqref{cauchy-J} in $(\tilde \Omega \setminus \tilde{\Gamma})\times ]0,T[$ if for any $(u_0,t_0) \in (\tilde{\Omega} \setminus \tilde{\Gamma}) \times ]0,T[$ and $\varphi \in C^1(]0,T[)$ such that $f^\e(u_0,\cdot)-\varphi$ attains a local maximum point at $t_0 \in ]0,T[$, one has
\begin{eqnarray*}
\ddt \varphi(t_0) \leq - \babs{\nabla_{J_\e}^- f^\e(u_0,t)}_\infty + \tilde{P}(u_0).
\end{eqnarray*}
The function $f^\e$ is a viscosity sub-solution of \eqref{cauchy-J} in $\tilde \Omega_T$ if it satisfies moreover $f^\e(u,t)\le \tilde \psi(u)$ for all $(u,t) \in \partial \tilde{\Omega}_T$.

A lsc function $f^\e: \tilde{\Omega}_T \to \R$  is a viscosity super-solution of \eqref{cauchy-J} in $(\tilde \Omega \setminus \tilde{\Gamma})\times ]0,T[$ if for any $(u_0,t_0) \in (\tilde{\Omega} \setminus \tilde{\Gamma}) \times ]0,T[$ and $\varphi \in C^1(]0,T[)$ such that $f^\e(u_0,\cdot)-\varphi$ attains a local minimum point at $t_0$, one has
\begin{eqnarray*}
\ddt \varphi(t_0) \geq - \babs{\nabla_{J_\e}^- f^\e(u_0,t)}_\infty + \tilde{P}(u_0).
\end{eqnarray*}
The function $f^\e$ is a viscosity super-solution of \eqref{cauchy-J} in $\tilde \Omega_T$ if it satisfies moreover $f^\e(u,t)\ge \tilde \psi(u)$ for all $(u,t) \in \partial \tilde{\Omega}_T$.

Finally, a locally bounded function $f^\e : \tilde{\Omega}_T \to \R$ is a viscosity solution of \eqref{cauchy-J} in $(\tilde \Omega \setminus \tilde  \Gamma) \times ]0,T[$ (resp. in $\tilde \Omega_T$) if  $(f^\e)^*$ is a viscosity sub-solution and $(f^\e)_*$ is a viscosity super-solution of \eqref{cauchy-J} in $(\tilde \Omega \setminus \tilde{\Gamma}) \times ]0,T[$ (resp. in $\tilde \Omega_T$).
\end{defn}

We define barrier sub-solution and super-solution of \eqref{cauchy-J} in a similar way as we have done for the local case in Definition~\ref{def:barriercauchy}, just replacing by the non-local notion of viscosity sub- and super-solution defined above. 


We start by providing a comparison result for problem \eqref{cauchy-J}.
\begin{prop}[Comparison principle for \eqref{cauchy-J}]\label{lem-comparison-cont}
Suppose that assumptions~\ref{assum:Om}--\ref{assum:Gam} and \ref{assum:gpos} hold. Assume that $f^\e$ (resp. $g^\e$) is a bounded viscosity sub- (resp. super-) solution of \eqref{cauchy-J}. Then
\[
f^\e \le g^\e \quad {\rm in}\; \tilde \Omega_T.
\]
\end{prop}
\begin{proof}
We argue by contradiction and suppose that there exists some point $(z,s) \in \tilde{\Omega}_T $ such that 
\[
f^\e(z,s) - g^\e(z,s) > 0.
\]
For $\eta > 0$ sufficiently small, we introduce the function $\Psi_\eta: (u,t) \in \tilde{\Omega}_T \mapsto f^\e(u,t) - g^\e(u,s) - \frac{\eta}{T-t}$ and denote 
\[
M_{\eta} = \sup_{(u,t) \in \tilde{\Omega}_T} \Psi_\eta(u,t) . 
\]
By upper semi-continuity and compactness, $M_\eta$ is actually a maximum attained at some point on $\tilde{\Omega}_T$, say $(\tilde u^*,\tilde t^*)$. Moreover, from the positivity assumption, we have $M_{\eta} > 0$ for $\eta > 0$ small enough.

We now duplicate the time variable and consider, for $\gamma>0$, the function 
\[
\Psi_{\eta,\gamma}: (u,t,s) \in \tilde{\Omega} \times [0,T]^2 \mapsto f^\e(u,t) - g^\e(u,s) - \frac{|t-s|^2}{2\gamma} - \frac{\eta}{T-t} ,
\]
and we denote
\[
M_{\gamma,\eta} = \sup_{(u,t,s) \in \tilde{\Omega} \times [0,T]^2} \Psi_{\eta,\gamma}(u,t,s). 
\]
Again, upper semi-continuity and compactness entails that the supremum is actually a maximum which is attained at some point $(\bar{u}_{\gamma},\bar{t}_{\gamma}, \bar{s}_{\gamma}) \in \tilde{\Omega}\times[0,T]^2$. We also have for $\eta$ sufficiently small
\[
M_{\gamma,\eta} \geq \Psi_{\eta,\gamma}(\tilde u^*,\tilde t^*,\tilde t^*) = \Psi_{\eta}(\tilde u^*,\tilde t^*) = M_{\eta} > 0 .
\]

Using classical arguments (see, e.g., \cite[Lemma~5.2]{barles11}), we deduce that there exists $(u^*,t^*)\in \tilde{\Omega} \times [0,T[$ such that
\begin{equation}\label{eq:limMgeta}
\begin{cases}
\bar{u}_{\gamma} \to u^* & \text{ as } \gamma \to 0, \\
\bar{t}_{\gamma},\; \bar{s}_{\gamma} \to t^* & \text{ as } \gamma \to 0, \\
\Psi_{\eta}(u^*,t^*) = M_\eta .
\end{cases}
\end{equation}
Note that if $t^*=0$, then we would have
\[
0<M_\eta=f^\e(u^*,0)-g^\e(u^*,0)-\frac \eta {T}\le \tilde \psi(u^*)-\tilde \psi(u^*) = 0 ,
\]
which is absurd. Hence $t^*>0$ which, in view of \eqref{eq:limMgeta}, implies that $\bar{t}_{\gamma}>0$ for $\gamma$ small enough. Moreover, if $\bar{u}_{\gamma} \in \tilde{\Gamma}$, then
\[
0<M_{\gamma,\eta}\le f^\e(\bar{u}_{\gamma},\bar{t}_{\gamma})-g^\e(\bar{u}_{\gamma},\bar{s}_{\gamma})\le \tilde \psi(\bar{u}_{\gamma})-\tilde \psi(\bar{u}_{\gamma}) = 0 ,
\]
which is again absurd. Hence $\bar{u}_{\gamma} \in \tilde \Omega\setminus \tilde\Gamma$.
%

The function $t \mapsto g^\e(\bar{u}_{\gamma},\bar{s}_{\gamma})+\frac{|t-\bar{s}_{\gamma}|^2}{2\gamma} +  \frac{\eta}{T-t}$ is smooth at $\bar{t}_{\gamma} > 0$, and reaches a maximum in $\bar{t}_{\gamma}$. Since $f^\e$ is a viscosity sub-solution of \eqref{cauchy-J}, we deduce that
\[
\frac{\bar{t}_{\gamma} - \bar{s}_{\gamma}}{\gamma} + \frac \eta{T^2} +  \babs{\nabla_{J_\e}^- f^\e(\bar{u}_{\gamma},\bar{t}_{\gamma})}_\infty - \tilde{P}(\bar{u}_{\gamma}) \leq 0 .
\]
Arguing in the same way, but now on $g^\e$, and using it is a viscosity super-solution of \eqref{cauchy-J}, we get that
\[
\frac{\bar{t}_{\gamma} - \bar{s}_{\gamma}}{\gamma}  +  \babs{\nabla_{J_\e}^- g^\e(\bar{u}_{\gamma},\bar{s}_{\gamma})}_\infty - \tilde{P}(\bar{u}_{\gamma}) \geq 0.
\]
Subtracting the last two inequalities, we obtain
\begin{align}\label{est-ct}
\frac \eta {T^2}\le&  \babs{\nabla_{J_\e}^- g^\e(\bar{u}_{\gamma},\bar{s}_{\gamma})}_\infty -
\babs{\nabla_{J_\e}^- f^\e(\bar{u}_{\gamma},\bar{t}_{\gamma})}_\infty .
\end{align}
We now use the fact that $(\bar{u}_{\gamma},\bar{t}_{\gamma},\bar{s}_{\gamma})$ is a maximum point of $\Psi_{\gamma,\eta}$ defined above. This implies in particular that
\[
f^\e(v,\bar{t}_{\gamma}) - g^\e(v,\bar{s}_{\gamma}) -\frac{|\bar{t}_{\gamma}-\bar{s}_{\gamma}|^2}{2\gamma} - \frac \eta{T-\bar{t}_{\gamma}}
\leq 
f^\e(\bar{u}_{\gamma},\bar{t}_{\gamma}) - g^\e(\bar{u}_{\gamma},\bar{s}_{\gamma}) -\frac{|\bar{t}_{\gamma}-\bar{s}_{\gamma}|^2}{2\gamma}-  \frac \eta{T-\bar{t}_{\gamma}} ,
\]
whence we get
\[
g^\e(\bar{u}_{\gamma},\bar{s}_{\gamma})-g^\e(v,\bar{s}_{\gamma})\le f^\e(\bar{u}_{\gamma}, \bar{t}_{\gamma})-f^\e(v,\bar{t}_{\gamma}).
\]
Multiplying both sides of this inequality by $J_\e$, which is non-negative by \ref{assum:gpos}, taking the maximum over $v \in \tilde{\Omega}$ and recalling \eqref{eq:normnablaJe}, \eqref{est-ct} becomes
\begin{align*}
\frac \eta {T^2} \le 0
\end{align*}
leading to a contradiction. 
\end{proof}

In the same vein as for problem \eqref{cauchy}, the following assumption is intended to impose compatibility properties between \eqref{cauchy-J} and the boundary conditions on $\partial\tilde\Omega_T$:\\
\fbox{\parbox{0.975\textwidth}{
\begin{enumerate}[label=({\textbf{H.\arabic*}}),start=11]
\item There exists $\tilde \psi_b\in \LIP{\tcb\Omega}$, with $\tilde \psi_b(u)=\tilde \psi(u)$ for all $u\in \tilde{\Gamma}$, such that $\tilde \psi_b$ is a sub-solution of \eqref{cauchy-J} in $\tilde \Omega_T$. \label{assum:psisssol-J}
\end{enumerate}}}\\

\begin{rem}
We refer to Remark~\ref{rem:asspsi} for a discussion on this assumption, which is similar to the one made in the local case. For instance, \ref{assum:psisssol-J} holds when $\tilde{\psi}=0$ and $\tilde{P} \geq 0$. This example, when considered on weighted graphs (see Section~\ref{sec:eikconvgraphs}), corresponds to computing distances on discrete images, meshes, point clouds, or any data that can be represented as a weighted graph; see \cite{Toutain16,Desquesnes17} and references therein.
\end{rem}
%

We are ready to provide an existence result. As for the local case, the proof is based on Perron's method and on the construction of barriers.
\begin{prop}[Existence result for \eqref{cauchy-J}]\label{pro:existence-J}
Suppose that assumptions~\ref{assum:Om}--\ref{assum:psi}, \ref{assum:gpos}--\ref{assum:gdec} and \ref{assum:psisssol-J} hold. \tcb{Then, problem \eqref{cauchy-J} admits a unique viscosity solution $f^\e$ (which is in fact continuous).} Moreover, there exists a function $\widebar{f}^\e \in \LIP{\tilde{\Omega}}$ such that
\begin{equation}\label{eq:barrier-f-J}
\tilde \psi_b\le f^\e\le \widebar{f}^\e\quad {\rm in }\;\tilde \Omega_T.
\end{equation}
\end{prop}

\begin{rem}
A close inspection of the forthcoming proof reveals that the Lipschitz constant estimate of the barrier super-solution $\widebar{f}^\e$ depends on the minimal distance between two points of $\tilde \Omega$ and can be very large. This is rather pessimistic but seems the price to pay to construct a barrier super-solution. On the other hand, and fortunately, this estimate will not enter our error bounds. 
\end{rem}

\begin{rem}
The authors of \cite{Oberman15,Desquesnes17} proved existence and uniqueness of the solution (not a viscosity one) in the special case of \eqref{eq:eikgraph}.
\end{rem}

\begin{proof}
The proof follows the same lines as the one of Proposition~\ref{pro:existence}, but adapted to the non-local setting. By assumption \ref{assum:psisssol-J}, $\tilde \psi_b$ is a barrier sub-solution of \eqref{cauchy-J}. We then have to construct a barrier super-solution. Existence will then be a direct consequence of the Perron's method while uniqueness and continuity will be direct consequences of the comparison principle provided in Proposition~\ref{lem-comparison-cont}.
%

Let
\[
\widebar{f}_1^\e(u,t)= \tilde \psi(u)+K_1t \qandq \widebar{f}_2^\e(u,t)=\tilde  \psi(u)+K_2d(u,\tilde{\Gamma}), \quad (u,t) \in \tilde{\Omega}_T ,
\]
where $K_1=\norm{\tilde{P}}_{L^\infty(\tilde{\Omega}\setminus\tilde{\Gamma})}$ and $K_2$ large enough to be determined. We then define 
\begin{equation}\label{eq:barf}
\widebar{f}^\e(u,t)=\min(\widebar{f}_1^\e(u,t),\widebar{f}_2^\e(u,t)) . 
\end{equation}
We will show that $\widebar{f}^\e$ is a barrier super-solution of \eqref{cauchy-J}. Arguing similarly to the proof of Proposition~\ref{pro:existence}, one has that $\widebar{f}^\e$ is (Lipschitz) continuous, hence lsc, and that the limit property required in Definition~\ref{def:barriercauchy} holds since for $u \in \tilde{\Gamma}$, we have 
\[
\widebar{f}_2^\e(u,t)= \tilde \psi(u)\le\widebar{f}_1^\e(u,t), \quad \forall (u,t) \in \tilde{\Gamma} \times [0,T],
\]
and thus 
\[
\widebar{f}^\e(u,t) = \tilde \psi(u), \quad \forall (u,t) \in \tilde{\Gamma} \times [0,T].
\]
It remains to show that $\widebar{f}^\e$ is a super-solution on $(\tilde{\Omega}\setminus \tilde{\Gamma})\times ]0,T[$ for $K_2$ large enough. 

Let $\NeigtGameta \eqdef \set{u \in \tilde \Omega:~d(u,\tilde{\Gamma}) \leq \eta}$, for $\eta$ small enough to be chosen shortly. Taking $K_2 \geq K_1T/\eta$, we have for any $u \in \tilde{\Omega} \setminus \NeigtGameta$ 
\[
\widebar{f}_2^\e(u,t) \geq \tilde{\psi}(x) + K_2 \eta \geq \tilde{\psi}(u) + K_1 T  \geq \widebar{f}_1^\e(u,t) .
\]
In turn, $\widebar{f}^\e=\widebar{f}_1^\e$ on $\tilde{\Omega} \setminus \NeigtGameta \times [0,T]$. Let $\varphi\in C^1(]0,T[)$ and $u_0 \in \tilde{\Omega} \setminus \tilde{\Gamma} \times [0,T]$ such that $\widebar{f}^\e(u_0,\cdot)-\varphi$ attains a local minimum at some $t_0 \in ]0,T[$.

If $u_0 \in \tilde{\Omega} \setminus \NeigtGameta$, then $\widebar{f}^\e(u_0,t_0)=\widebar{f}^\e_1(u_0,t_0)$. One easily shows following the same steps as for $\widebar{f}_1$ in the proof of Proposition~\ref{pro:existence}, that
\[
\ddt \varphi(t_0) = K_1.
\]
It then follows that
\[
\ddt \varphi(t_0) + |\nabla_{J^\e}^-\widebar{f}^\e_1(u_0,t_0)|_\infty-\tilde P(u) \geq K_1 - \tilde P(u) \geq K_1 - \norm{\tilde{P}}_{L^\infty(\tilde{\Omega}\setminus\tilde{\Gamma})}  = 0.
\]

If $u_0 \in \NeigtGameta \setminus \tilde{\Gamma}$, we have two cases. Either $\widebar{f}^\e(u_0,t_0) = \widebar{f}_1^\e(u_0,t_0)$, and we are done, or $\widebar{f}^\e(u_0,t_0) = \widebar{f}_2^\e(u_0,t_0)$. In this case, we have (see again the proof of Proposition~\ref{pro:existence}) that
\[
\ddt \varphi(t_0) = 0 ,
\]
and thus, for every $v\in \tilde \Gamma$, we have
\begin{align}
\ddt \varphi(t_0)+|\nabla_{J^\e}^-\widebar{f}^\e_2(u_0,t_0)|_\infty-\tilde P(u_0)
&\ge J_\e(u_0,v)(\widebar{f}_2^\e(u_0,t_0)-\widebar{f}_2^\e(v,t_0))-\tilde P(u_0) \nonumber \\
&=\frac{1}{\e} C_g^{-1}g\pa{\frac{|u_0-v|}{\e}}(\tilde{\psi}(u_0)+K_2d(u_0,\tilde{\Gamma})-\tilde \psi(v)) - \tilde P(u_0) \label{eq:ineqf2} .
\end{align}
Denoted by $\tilde d_0$ the minimal distance between two points of $\tilde \Omega$ (note that, since $u_0\in \NeigtGameta \setminus \tilde{\Gamma}$, $\tilde d_0\le \eta$), we get that there exists $v_0 \in \tilde{\Gamma}$ such that 
\[
d(u_0,\tilde{\Gamma})=|u_0-v_0| \in [\tilde d_0,\eta] .
\]
Since $\tilde{\Gamma} \subset \tilde{\Omega}$, and in view of Remark \ref{assum:gubnd} and \ref{assum:gdec}, we can choose $K_2 \geq C_g c_g^{-1}\e \tilde d_0^{-1} \norm{\tilde{P}}_{L^\infty(\tilde{\Omega}\setminus\tilde{\Gamma})} + \Lip{\tilde{\psi}}$ and $\eta = a\e$ (recall the definition of $a$ from assumption \ref{assum:gdec}). Then continuing from \eqref{eq:ineqf2}, and using \ref{assum:gdec}, we get
\begin{align}\label{eq:13}
\ddt \varphi(t_0)+|\nabla_{J^\e}^-\widebar{f}^\e_2(u_0,t_0)|_\infty-\tilde P(u_0)
&\geq C_g^{-1}\frac{|u_0-v_0|}{\e}g\pa{\frac{|u_0-v_0|}{\e}}(K_2-\Lip{\tilde{\psi}}) - \tilde{P}(u_0)\\ \nonumber
&\geq C_g^{-1}\frac{\tilde d_0}{\e} g(a)(K_2-\Lip{\tilde{\psi}}) - \norm{\tilde{P}}_{L^\infty(\tilde{\Omega}\setminus\tilde{\Gamma})}\\ \nonumber
&= C_g^{-1} c_g \frac{\tilde d_0}{\e} (K_2-\Lip{\tilde{\psi}}) - \norm{\tilde{P}}_{L^\infty(\tilde{\Omega}\setminus\tilde{\Gamma})} \\
&\geq 0 .\nonumber
\end{align}

To summarize, taking $\eta =a\e$ and $K_2 \geq \max\bpa{C_g c_g^{-1} \e \tilde d_0^{-1} \norm{\tilde{P}}_{L^\infty(\tilde{\Omega}\setminus\tilde{\Gamma})} + L_{\tilde \psi},K_1T/\eta}$, we conclude that the desired super-solution inequality is satisfied in all cases. We then conclude that $\widebar{f}^\e$ is indeed a barrier super-solution as claimed. Existence and uniqueness then follow from Perron's method and the comparison principle.

%
%
\end{proof}

We now establish regularity properties for the solution of \eqref{cauchy-J}.
\begin{thm}\label{lip-viscosity-J}
Suppose that assumptions~\ref{assum:Om}--\ref{assum:psi}, \ref{assum:gpos}--\ref{assum:gdec} and \ref{assum:psisssol-J} hold. Let $f^\e$ be the bounded continuous viscosity solution of \eqref{cauchy-J}. Then 
\begin{eqnarray}
f^\e(u,\cdot) \in \LIP{[0,T[} \qwithq \Lip{f^\e(u,\cdot)} \leq L, \quad \forall u \in \tilde{\Omega}, \label{lip-t-ct}
\end{eqnarray}
where 
\[
L = \Lip{\tilde{\psi}} + \norm{\tilde{P}}_{L^\infty(\tilde{\Omega}\setminus\tilde{\Gamma})} .
\]
Moreover, for all $(u,v) \in \tilde{\Omega}^2$ and $t \in [0,T[$ such that $|u-v| \leq a\e$, where $a$ is defined in \ref{assum:gdec}, we have
\begin{eqnarray}
\abs{f^\e(u,t) - f^\e(v,t)} \leq c_g^{-1}C_g \bpa{L + \norm{\tilde{P}}_{L^\infty(\tilde{\Omega}\setminus\tilde{\Gamma})}}\e .
\label{lip-t-space}
\end{eqnarray}
Assume also that for $(u,v)\in \tilde \Omega^2$, there exists $k(\e)\in\N$ and a path $(u_1=u,u_2,\dots, u_{k(\e)}=v)$ with $|u_{i+1}-u_i| \le a \e$, $i=1,\dots,k(\e)-1$.
Then for all $t \in [0,T[$, we have
\begin{eqnarray}
\abs{f^\e(u,t) - f^\e(v,t)} \leq  c_g^{-1}C_g\bpa{L + \norm{\tilde{P}}_{L^\infty(\tilde{\Omega}\setminus\tilde{\Gamma})}} k(\e)\e.
\label{eq:globlip-space-J-1}
\end{eqnarray}
\end{thm}


\begin{proof}
For $u \in \tilde{\Gamma}$, \eqref{lip-t-ct} trivially holds. We consider hereafter $u \in \tilde{\Omega} \setminus \tilde{\Gamma}$, and we first show that for any $t \in [0,T[$,
\begin{eqnarray}\label{compar-ineq-ct}
|f^\e(u,t) - f^\e(u,0)| \leq L t.
\end{eqnarray}
We define for $(u,t) \in \tilde{\Omega}_T$
\[
f_1^\e(u,t) = \tilde{\psi}(u) - Lt; \qquad  f_2^\e(u,t) = \tilde{\psi}(u) + Lt.
\]
We claim that $f_1^\e$ (resp. $f_2^\e$) is a sub-solution (resp. super-solution) of \eqref{cauchy-J}. Since $f_1^\e$ and $f_2^\e$ are smooth in time, it's enough to prove it pointwise.

We have $f_1^\e \leq \tilde{\psi}$ on $\partial\tilde{\Omega}_T$, and for all $(u,t) \in (\tilde{\Omega} \setminus \tilde{\Gamma}) \times ]0,T[$, 
\begin{equation}\label{sub-2-ct}
\begin{aligned}
\ddt f_1^\e(u,t)+\babs{\nabla_{J_\e}^- f_1^\e(u,t)}_\infty - \tilde{P}(u)
&=-L+\max_{v\in\tilde{\Omega}} (\e C_g)^{-1}g\pa{\frac{|u-v|}{\e}}  (\tilde{\psi}(u) - \tilde{\psi}(v)) - \tilde{P}(u) \\
&\leq -L+L_{\tilde{\psi}}\max_{v\in\tilde{\Omega}} C_g^{-1}\frac{|u-v|}{\e} g\pa{\frac{|u-v|}{\e}}   + \normL{\tilde{P}}_{L^\infty(\tilde{\Omega} \setminus \tilde{\Gamma})} \\
&\leq 0 ,
\end{aligned}
\end{equation}
where we used \ref{assum:psi} in the first inequality and Remark \ref{assum:gubnd} in the last one. Therefore, this shows our claim on $f_1^\e$. A similar argument shows also that $f_2^\e$ is a super-solution of \eqref{cauchy-J}.

Now, for any $(u,t) \in \partial \tilde{\Omega}_T$, we have
\begin{eqnarray*}
f_1^\e(u,t) \leq \tilde{\psi}(u)=f^\e (u,t) \leq f_2^\e(u,t).
\end{eqnarray*}
Hence, since $f_1^\e$ and $f_2^\e$ are bounded and continuous (by assumption on $\tilde{\psi}$), and so is $f^\e$, applying the comparison principle of Proposition~\ref{lem-comparison-cont} twice yields that for any $(u,t) \in \tilde{ \Omega} \times [0,T[$,
\begin{eqnarray*}
f^\e(u,0) - Lt = \tilde{\psi}(u) - Lt \leq f^\e(u,t) \leq  \tilde{\psi}(u) + Lt = f^\e(u,0) + Lt ,
\end{eqnarray*}
which shows \eqref{compar-ineq-ct}.
We now apply this estimate to prove \eqref{lip-t-ct}. Let $h>0$ sufficiently small. We have that $f^\e$ is a solution of \eqref{cauchy-J} with initial condition $f^\e(\cdot,0)$ and $f^\e(\cdot,\cdot+h)$ is also a solution of \eqref{cauchy-J} with initial condition $f^\e(\cdot,h)$. Applying again the comparison principle of Proposition~\ref{lem-comparison-cont} and using \eqref{compar-ineq-ct}, we obtain for any $(u,t) \in \tilde{\Omega} \times [0,T[$,
\begin{eqnarray*}
|f^\e(u,t+h) - f^\e(u,t)| 
&\leq& |f^\e(u,h) - f^\e(u,0)|\\
&\leq& L h. 
\end{eqnarray*}
Passing to the limit as $h \to 0$ yields the desired time regularity claim.

Let us turn to the space regularity estimate \eqref{lip-t-space}. Let $(u,t) \in \tilde \Omega_T$. If $u\in \partial \tilde \Omega_t$, then
\[
f^\e(u,t)-f^\e(v,t)\le \psi_b(u)-\psi_b(v)\le L_{\psi_b}|u-v|
\]
and \eqref{lip-t-space} holds. Assume now that $(u,t)\in (\tilde{\Omega}\setminus\tilde{\Gamma}) \times ]0,T[$ is such that $f^\e$ is differentiable in time at $(u,t)$. For such points, we have from \eqref{cauchy-J} and \eqref{lip-t-ct} that
\[
\abs{\nabla_{J_\e}^- f^\e(u,t)}_\infty \leq  L+\norm{\tilde{P}}_{L^\infty(\tilde{\Omega}\setminus\tilde{\Gamma})} .
\] 
Let $v\in \tilde \Omega$ be such that $|u-v|\le a\e$. 
We then have, recalling \ref{assum:gdec}, that
\begin{align*}
 c_g (\e C_g)^{-1}  \pa {f^\e(u,t)-f^\e(v,t)}
&\leq (\e C_g)^{-1} g\pa{\frac{|u-v|}\e}\pa{f^\e(u,t)-f^\e(v,t)}  \\
&= J_\e(u,v)\pa{f^\e(u,t)-f^\e(v,t)} \\
&\leq \abs{\nabla_{J_\e}^- f^\e(u,t)}_\infty\\
&\leq \bpa{L+\norm{\tilde{P}}_{L^\infty(\tilde{\Omega}\setminus\tilde{\Gamma})}}.
\end{align*}
Exchanging the role of $u$ and $v$, we get the result.

The global estimate is now a direct consequence of \eqref{lip-t-space}. Indeed, we have
\begin{align*}
\babs{f^\e(u,t) - f^\e(v,t)} \leq \sum_{i=1}^{k(\e)-1}\babs{f^\e(u_{i+1},t) - f^\e(u_i,t)} 
&\leq c_g^{-1}C_g \bpa{L + \norm{\tilde{P}}_{L^\infty(\tilde{\Omega}\setminus\tilde{\Gamma})}} \sum_{i=1}^{k(\e)-1} \e \\
&\leq c_g^{-1}C_g \bpa{L + \norm{\tilde{P}}_{L^\infty(\tilde{\Omega}\setminus\tilde{\Gamma})}}  k(\e) \e .
\end{align*}
\end{proof}

The following lemma gives a sufficient condition under which the requirements of the global estimate of Theorem~\ref{lip-viscosity-J} hold true.

\begin{lem}\label{lip-viscosity-J-1}
Suppose that assumptions~\ref{assum:Om}--\ref{assum:psi}, \ref{assum:gpos}--\ref{assum:gdec} and \ref{assum:psisssol-J} hold. Let $f^\e$ be the bounded continuous viscosity solution of \eqref{cauchy-J}.
Assume also that 
\begin{equation}\label{eq:compatdomains}
\max_{x \in \Omega} d(x,\tilde{\Omega}) < a\e/(4\sqrt{m}).
\end{equation}
Then for all $(u,v) \in \tilde{\Omega}^2$ and $t \in [0,T[$, the following holds
\begin{eqnarray}
\abs{f^\e(u,t) - f^\e(v,t)} \leq K\pa{|u-v|+\e} ,
\label{eq:globlip-space-J}
\end{eqnarray}
where $K=2c_g^{-1}C_g\bpa{L + \norm{\tilde{P}}_{L^\infty(\tilde{\Omega}\setminus\tilde{\Gamma})}}m^{3/2}$.
\end{lem}
\begin{proof}
We use a discretization argument\footnote{A similar argument is implicitly underlying the proof of \cite[Lemma~15]{Calder19} for the special case where $\Omega$ is the flat torus and $\tilde{\Omega}$ is discrete.} through the notion of $\delta$-nets. Consider $\Omega$ as a metric space endowed with the metric induced by the $|\cdot|_{\infty}$-norm. A $\delta$-net of $\Omega$ is a set $\set{x_1,x_2,\ldots,x_N} \eqdef S_\delta \subset \Omega$ such that for all $x \in \Omega$, there exists $y \in S_\delta$ such that $|x-y|_\infty \leq \delta$. This is equivalent here to saying that $\Omega$ can be covered by hypercubes of side length $2\delta$ centered at the points in $S_\delta$. It is known that $\Omega$ is compact if and only if $S_\delta$ is finite.

Choose $\delta = a\e/(4\sqrt{m})$. Thus, using that $S_\delta \subset \Omega$, we get
\begin{align*}
\max_{x \in S_\delta}\min_{y \in \tilde{\Omega}} |x-y|_\infty 
&\leq \max_{x \in S_\delta}d(x,\tilde{\Omega}) \\
&\leq \max_{x \in \Omega}d(x,\tilde{\Omega}) = \distH(\Omega,\tilde{\Omega}),
\end{align*}
where the last identity follows from the fact that $\tilde{\Omega} \subset \Omega$. It then follows from \eqref{eq:compatdomains} that
\begin{align*}
\max_{x \in S_\delta}\min_{y \in \tilde{\Omega}} |x-y|_\infty < a\e/(4\sqrt{m}) ,
\end{align*}
whence we deduce that each hypercube of the $\delta$-covering contains at least one point in $\tilde{\Omega}$. This in turn entails that for any $(u,v) \in \tilde{\Omega}^2$ that belong to two horizontally or vertically adjacent hypercubes in the $\delta$-covering, centered say at respectively $x_i$ and $x_j$ in $S_\delta$, one has
\[
\babs{u-v} \leq \babs{u-x_i} + \babs{x_i-x_j} + \babs{v-x_j} \leq \sqrt{m}\pa{\delta + 2\delta + \delta} = a\e .
\]
This allows to infer that for any $(u,v) \in \tilde{\Omega}^2$, there exists a path $(u_1=u,u_2,u_3,\ldots,u_k=v)$, where $u_i \in \tilde{\Omega}$ and $|u_{i+1}-u_i| \leq a\e$ for all $i$. Moreover, we have the simple estimate
\[
k \leq m \e^{-1}\pa{2\sqrt{m}|u-v| + \e} .
\]
Injecting this in \eqref{eq:globlip-space-J-1}, we get the result.
\end{proof}

\begin{rem}\label{rem:compatdomainker}
A consequence of the proof of Lemma \ref{lip-viscosity-J-1} is that, under assumption \eqref{eq:compatdomains}, since $a\le r_g$, we have
\begin{equation}\forall u \in \tilde{\Omega}, \exists v \in \tilde{\Omega}, v \neq u \text{ such that } |u-v| \in \e \supp(g).
\label{assum:compatdomainker}
\end{equation}
This assumption is quite natural. It basically avoids that the non-local operator $\babs{\nabla_{J_\e}^- f^\e(u,s)}_\infty$ is trivially zero for all $u \in \tilde{\Omega}$ when $\e$ is too small. In particular, as $\tilde{\Omega}$ is finite, this condition imposes that $\tilde{\Omega}$ has to fill out $\Omega$ at least as fast as the rate at which $\e$ goes to $0$.
\end{rem}

%
%

\section{Consistency and error bounds}
\label{sec:main}

\subsection{Continuous time non-local to local error bound}\label{subsec:eikconvcont}
In this section we provide an estimate that compares viscosity solutions of \eqref{cauchy-J} and \eqref{cauchy}. This estimate will be instrumental to derive the remaining error bounds. For this, we need to strenghthen \eqref{eq:compatdomains} by assuming:
\\
\noindent\fbox{\parbox{0.975\textwidth}{
\textup{
\begin{enumerate}[label=({\textbf{H.\arabic*}}),itemindent=5ex,start=12]
\item $\max_{x \in \Omega} d(x,\tilde{\Omega}) \leq a\e^{1+\nu}/(4\sqrt{m})$, $\nu > 0$. \label{assum:compatdomains}
\end{enumerate}}}}\\

\begin{thm}\label{thm:continuous-time-estimate}
Let $T>0$, $\e_0=\min(1/(2r_g)^2,1)$ and $\e \in ]0,\e_0]$. Suppose that assumptions~\ref{assum:Om}--\ref{assum:compatdomains} hold, and let $f$ and $f^\varepsilon$ be the unique viscosity solutions of respectively \eqref{cauchy} and \eqref{cauchy-J}, given in Proposition~\ref{pro:existence} and Proposition~\ref{pro:existence-J}. Then, there exists  a constant $K > 0$ depending only on the dimension $m$, $\norm{\psi}_{L^\infty(\Omega)}$, $\norm{P}_{L^\infty(\Omega\setminus\Gamma)}$, $\Lip{\psi}$, $\Lip{\tilde{\psi}}$, $\Lip{P}$, $\Lip{\tilde{P}}$, $L_g$, $C_g$ and $c_g$ such that 
\begin{align*}
\normL{f^\e-f}_{L^\infty(\tilde{\Omega} \times [0,T[)} 
\leq& K(T+1)\left(\e^{\min(\nu,1/2)}
+ \normL{P - \tilde{P}}_{L^\infty(\tilde{\Omega}\setminus\tilde{\Gamma})}\right) + \normL{\psi - \tilde{\psi}}_{L^\infty(\tilde{\Omega})}
+ K\distH(\Gamma,\tilde \Gamma),
\end{align*}
In particular, if $\distH(\Gamma,\tilde \Gamma) = O(\e^{\min(\nu,1/2)})$, then
\begin{align*}
\normL{f^\e-f}_{L^\infty(\tilde{\Omega} \times [0,T[)} \leq K(T+1)\left(\e^{\min(\nu,1/2)}
+ \normL{P - \tilde{P}}_{L^\infty(\tilde{\Omega}\setminus\tilde{\Gamma})}\right)
+ \normL{\psi - \tilde{\psi}}_{L^\infty(\tilde{\Omega})} .
\end{align*}
The fastest convergence rate in $\e$ is then achieved when $\nu=1/2$ provided that $\distH(\Omega,\tilde \Omega) = O(\e^{3/2})$ and $\distH(\Gamma,\tilde \Gamma) = O(\e^{1/2})$. 
\end{thm}

\begin{proof}
The idea of the proof is inspired by that in \cite{cl84} and revisited in \cite{nf08} for non-local equations. The main difficulties in our case come from the fact that the sets and the boundaries in the two equations are different and from the fact that $\tilde \Omega$ is finite. In the following, $K$ denotes any positive constant that depends only on the data, but may change from one line to another.

\tcb{The proof is divided into three steps:
\begin{itemize}
\item In the first step, we use the classical doubling of variables argument by introducing the test-function  
\begin{equation}\label{varphi-ct}
\Psi_{\gamma,\eta}(u,s,x,t) = f^\epsilon(u,s) - f(x,t) - \frac{|x-u|^2}{2\gamma} -\frac{|t-s|^2}{2\gamma} - \eta s,
\end{equation}
for $\gamma > 0$ and $\eta > 0$, and we show that the maximum point is attained at $(\bar{u},\bar{s},\bar{x},\bar{t}) \in \tilde{\Omega}_T \times \Omega_T$. We also show some estimates on this point of maximum. 
\item In the second step, we show that the maximum point $(\bar{u},\bar{s},\bar{x},\bar{t})$ is actually achieved on the {\textit{boundary}} for $\eta$ large enough; we have either $(\bar{u},\bar{s}) \in \mathcal N^\alpha_\Gamma\times [0,T[\ \cup\  \tilde \Omega\times \{0\}$  or $(\bar{x},\bar{t}) \in \mathcal N^\alpha_\Gamma\times [0,T[\ \cup\  \Omega\times \{0\}$. This step is the most complicated and the most technical one. 
\item Finally, in the last step, we show, for $\eta$ large enough and using the estimates over the maximum points, the claimed bound.
\end{itemize}
}

\begin{enumerate}[label=Step~\arabic*.]
\item {\em Test-function and maximum point.}
%
%
%

\noindent 
%


The function $\Psi_{\gamma,\eta}$ being continuous (since $f$ and $f^\e$ are) on the compact set $\tilde{\Omega}_{T} \times \Omega_T$ (see \ref{assum:Om}), it achieves its maximum at a point which we denote by $(\bar{u},\bar{s},\bar{x},\bar{t})\in  \tilde{\Omega}_{T} \times \Omega_T$. At this point, we have $\Psi_{\gamma,\eta}(\bar{u},\bar{s},\bar{x},\bar{t}) \geq \Psi_{\gamma,\eta}(\bar{u},\bar{s},\bar{u},\bar{t})$ since $\bar{u} \in \tilde{\Omega} \subset \Omega$ by \ref{assum:Om}. This implies, in view of \eqref{eq:lipspace} (see Theorem~\ref{lip-viscosity}), that
\[
\frac{|\bar x - \bar u|^2}{2\g}\le f(\bar u,\bar t)-f(\bar x,\bar t)\le K |\bar x - \bar u|,
\]
and thus,
\begin{equation}\label{lip-pro-x-ct}
|\bar x - \bar u| \leq K\gamma .
\end{equation}
In the same way, using that $ \Psi_{\gamma,\eta}(\bar{u},\bar{s},\bar{x},\bar{t}) \geq \Psi_{\gamma,\eta}(\bar{u},\bar{t},\bar{x},\bar{t}) $, we get using \eqref{lip-t-ct},
\begin{equation}\label{lip-pro-t-ct}
|\bar{t} - \bar{s}| \leq (K+\eta) \gamma.
\end{equation}

\item {\em Excluding interior points from the maximum.}

\noindent We fix  $\a=\e^{1/2}$ so that $\e<\a/r_g$ for $\e\le \e_0$. We show that for $\eta$ large enough, we have either $(\bar{u},\bar{s}) \in \mathcal N^\a_\Gamma\times [0,T[\ \cup\  \tilde \Omega\times \{0\}$  or $(\bar{x},\bar{t}) \in \mathcal N^\a_\Gamma\times [0,T[\ \cup\  \Omega\times \{0\}$. We argue by contradiction, assuming that $(\bar u,\bar s)\in (\tilde \Omega \setminus \mathcal N^\a_\Gamma) \times ]0,T[$ and $(\bar x,\bar t)\in (\Omega\setminus \mathcal N^\a_\Gamma) \times ]0,T[$. Using that $\bar{s}$ is a maximum point of the function $s \mapsto \Psi_{\gamma,\eta}(\bar{u},s,\bar{x},\bar{t})$ and the fact that $f^\e$ is a viscosity sub-solution of \eqref{cauchy-J}, we get
\begin{equation}
\eta + \frac{\bar{s} - \bar{t}}{\gamma}  \leq -\babs{\nabla_{J_\e}^- f^\e(\bar{u},\bar{s})}_\infty + \tilde{P}(\bar{u}), \label{sub-ineq}
\end{equation}
In the same way, using that $(\bar{x},\bar{t})$ is a minimum point of the function $(x,t) \mapsto -\Psi_{\gamma,\eta}(\bar{u},\bar{s},x,t)$ and the fact that $f$ is a super-solution of \eqref{cauchy}, we get
\begin{equation}
\frac{\bar{s} - \bar{t}}{\gamma}  \geq - \frac{|\bar x - \bar u|}{\gamma} + P(\bar{x}) .
\label{super-ineq}
\end{equation}
Observe that $\forall v \in \tilde{\Omega}$, we have
\[
\Psi_{\gamma,\eta}(\bar{u},\bar{s},\bar{x},\bar{t}) - \Psi_{\gamma,\eta}(v,\bar{s},\bar{x},\bar{t}) = f^\e(\bar{u},\bar{s}) - f^\e(v,\bar{s}) + \frac {|\bar{x} - v|^2 - |\bar{x} - \bar{u}|^2}{2\g} .
\]
$(\bar{u},\bar{s},\bar{x},\bar{t})$ being a maximizer of $\Psi_{\gamma,\eta}$, we have for any $v \in \tilde{\Omega}$
\begin{align*}
2\g\pa{f^\e(\bar{u},\bar{s}) - f^\e(v,\bar{s})} 
&\geq |\bar{x} - \bar{u}|^2 - |\bar{x} - v|^2 \\
&= -|\bar{u}-v|^2 + 2\dotp{v-\bar{u}}{\bar{x}-\bar{u}} .
\end{align*}

It then follows that
\begin{equation}\label{ineq-2-ct}
\begin{aligned}
\babs{\nabla_{J_\e}^- f^\e(\bar{u},\bar{s})}_\infty
&= \max_{v\in\tilde{\Omega} \cap \ball_{\e r_g}(\bar{u})}J_\e(\bar{u},v) (f^\e(\bar{u},\bar{s}) -f^\e(v,\bar{s})) \\
&\geq (2\gamma)^{-1}\max_{v\in\tilde{\Omega} \cap \ball_{\e r_g}(\bar{u})}J_\e(\bar{u},v) \pa{-|\bar{u}-v|^2 + 2\dotp{v-\bar{u}}{\bar{x}-\bar{u}}} \\
&\geq (2\gamma)^{-1}\Bpa{2~~\underset{\mathrm{T}_1}{\underbrace{\max_{v\in\tilde{\Omega} \cap \ball_{\e r_g}(\bar{u})}J_\e(\bar{u},v)\dotp{v-\bar{u}}{\bar{x}-\bar{u}}}} -  \underset{\mathrm{T}_2}{\underbrace{\max_{v\in\tilde{\Omega} \cap \ball_{\e r_g}(\bar{u})}J_\e(\bar{u},v)|\bar{u}-v|^2}}} .
\end{aligned}
\end{equation}
To bound the term $\mathrm{T}_2$, we have, using \ref{assum:gpos}, Remark~\ref{assum:gubnd} and Remark~\ref{rem:compatdomainker} (see \eqref{assum:compatdomainker})
\begin{equation}\label{eq:T2}
\begin{aligned}
\mathrm{T}_2
&\leq \max_{v\in\ball_{\e r_g}(\bar{u})}J_\e(\bar{u},v)|\bar{u}-v|^2 \\
&= \max_{\tau \in [0,r_g]} \max_{|\bar{u}-v| = \e\tau} (\e C_g)^{-1}g\pa{\frac{|\bar{u}-v|}{\e}} |\bar{u}-v|^2 \\
&= \e \max_{\tau \in [0,r_g]} \tau^2 C_g^{-1}g(\tau) \leq r_g\e .
\end{aligned}
\end{equation}
Let us now turn to bounding $\mathrm{T}_1$. Observe that {$\mathrm{T}_1 \geq 0$} since $\bar{u} \in \tilde{\Omega} \cap \ball_{\e r_g}(\bar{u})$. We then decompose $\mathrm{T}_1$ as
\begin{multline}\label{eq:T1}
\mathrm{T}_1
=\underset{\mathrm{T}_3}{\underbrace{\max_{v\in\Omega \cap \ball_{\e r_g}(\bar{u})}J_\e(\bar{u},v)\dotp{v-\bar{u}}{\bar{x}-\bar{u}}}} \\
+ \Bpa{\max_{v\in\tilde{\Omega} \cap \ball_{\e r_g}(\bar{u})}J_\e(\bar{u},v)\dotp{v-\bar{u}}{\bar{x}-\bar{u}}
- \max_{v\in\Omega \cap \ball_{\e r_g}(\bar{u})}J_\e(\bar{u},v)\dotp{v-\bar{u}}{\bar{x}-\bar{u}}} .
\end{multline}
Note that $\mathrm{T}_3$ is also non-negative. Since $d(\bar{u},\Gamma)\ge \a$, we have, for $\e\le \e_0$, $\ball_{\e r_g}(\bar{u}) \subset \Omega$, and in turn,
\begin{equation}\label{eq:T3}
\begin{aligned}
\mathrm{T}_3
&=\max_{v\in\ball_{\e r_g}(\bar{u})}J_\e(\bar{u},v)\dotp{v-\bar{u}}{\bar{x}-\bar{u}} \\
&= \max_{\tau \in [0,r_g]} \max_{|\bar{u}-v| = \e\tau} (\e C_g)^{-1}g\pa{\frac{|\bar{u}-v|}{\e}}\dotp{v-\bar{u}}{\bar{x}-\bar{u}} \\
&= \max_{\tau \in [0,r_g]} (\e C_g)^{-1}g(\tau)\e\tau |\bar{x}-\bar{u}| = |\bar{x}-\bar{u}| .
\end{aligned}
\end{equation}
On the other hand, one can bound $\mathrm{T}_3$ from above as follows. Let $\beta=a\e^{1+\nu}/(4\sqrt{m})$. In view of \ref{assum:compatdomains}, we have $\Omega \subset \tilde{\Omega} + \ball_{\beta}(0)$, and thus for $\e$ small enough so that $\beta<\epsilon$ and using \ref{assum:glip}
\begin{equation}\label{eq:ubT3}
\begin{aligned}
0 \leq \mathrm{T}_3
&\leq \max_{v\in\pa{\tilde{\Omega}+\ball_{\beta}(0)} \cap \ball_{\e r_g}(\bar{u})}J_\e(\bar{u},v)\dotp{v-\bar{u}}{\bar{x}-\bar{u}} \\
&\leq \max_{v\in\pa{\tilde{\Omega} \cap \ball_{\e (r_g+1)}(\bar{u})}+\ball_{\beta}(0)} J_\e(\bar{u},v)\dotp{v-\bar{u}}{\bar{x}-\bar{u}} \\
&= (\e C_g)^{-1}\max_{v\in\tilde{\Omega} \cap \ball_{\e (r_g+1) }(\bar{u}),w \in \ball_{\beta}(0)} g\pa{\frac{|v-\bar{u}+w|}{\e}}\dotp{v-\bar{u}+w}{\bar{x}-\bar{u}} \\
&= (\e C_g)^{-1}\max_{v\in\tilde{\Omega} \cap \ball_{\e (r_g+1)}(\bar{u}),w \in \ball_{\beta}(0)} \pa{g\pa{\frac{|v-\bar{u}|}{\e}}+L_g\frac{|w|}{\e}}\dotp{v-\bar{u}+w}{\bar{x}-\bar{u}} \\
&=\max_{v\in\tilde{\Omega} \cap \ball_{\e (r_g+1)}(\bar{u}),w \in \ball_{\beta}(0)} \pa{J_\e(\bar{u},v)+L_gC_g^{-1}\frac{|w|}{\e^2}}\dotp{v-\bar{u}+w}{\bar{x}-\bar{u}} \\
&\leq \max_{v\in\tilde{\Omega} \cap \ball_{\e r_g}(\bar{u})} J_\e(\bar{u},v)\dotp{v-\bar{u}}{\bar{x}-\bar{u}}
+ \max_{v\in\tilde{\Omega} \cap \ball_{\e r_g}(\bar{u}),w \in \ball_{\beta}(0)} J_\e(\bar{u},v)\dotp{w}{\bar{x}-\bar{u}} \\
& + L_gC_g^{-1}\max_{v\in\tilde{\Omega} \cap \ball_{\e (r_g+1)}(\bar{u}),w \in \ball_{\beta}(0)} \frac{|w|}{\e^2}\dotp{v-\bar{u}}{\bar{x}-\bar{u}} 
+ L_gC_g^{-1}\max_{w \in \ball_{\beta}(0)} \frac{|w|}{\e^2}\dotp{w}{\bar{x}-\bar{u}} .
\end{aligned}
\end{equation}
We have using again \ref{assum:glip} and \eqref{lip-pro-x-ct},
\begin{equation}\label{eq:ubterm1T3}
\begin{aligned}
\max_{v\in\tilde{\Omega} \cap \ball_{\e r_g}(\bar{u}),w \in \ball_{\beta}(0)} J_\e(\bar{u},v)\dotp{w}{\bar{x}-\bar{u}} 
&=\Bpa{\max_{v\in\tilde{\Omega} \cap \ball_{\e r_g}(\bar{u})} J_\e(\bar{u},v)}\Bpa{\max_{w \in \ball_{\beta}(0)} \dotp{w}{\bar{x}-\bar{u}}} \\
&\leq \beta|\bar{x}-\bar{u}| \max_{\tau \in [0,r_g], |v-\bar{u}|=\tau\e} 
(\e C_g)^{-1} \pa{g(0) + L_g\frac{|v-\bar{u}|}{\e}} \\
&\leq K\frac{\beta}{\e}\gamma \leq K\e^\nu\gamma .
\end{aligned}
\end{equation}
Similar computation gives
\begin{equation}\label{eq:ubtermsT3}
\begin{aligned}
L_gC_g^{-1}\max_{v\in\tilde{\Omega} \cap \ball_{\e (r_g+1)}(\bar{u}),w \in \ball_{\beta}(0)} \frac{|w|}{\e^2}\dotp{v-\bar{u}}{\bar{x}-\bar{u}} 
&\leq K \frac{\beta}{\e}\gamma \leq K \e^\nu\gamma \qandq \\
L_gC_g^{-1}\max_{w \in \ball_{\beta}(0)} \frac{|w|}{\e^2}\dotp{w}{\bar{x}-\bar{u}} 
&\leq K \pa{\frac{\beta}{\e}}^2\gamma \leq K \e^\nu\gamma .
\end{aligned}
\end{equation}
Plugging \eqref{eq:ubterm1T3} and \eqref{eq:ubtermsT3} into \eqref{eq:ubT3}, and then combining with \eqref{eq:T3} and \eqref{eq:T1}, we get
\[
\mathrm{T}_1 \geq |\bar{x}-\bar{u}| - K \e^\nu\gamma .
\]
Injecting this and \eqref{eq:T2} into \eqref{ineq-2-ct}, we arrive at
\begin{equation*}
\babs{\nabla_{J_\e}^- f^\e(\bar{u},\bar{s})}_\infty \geq \frac{|\bar{x}-\bar{u}|}{2\gamma} - 
K \pa{\e^\nu + \frac{\e}{\gamma}} .
\end{equation*}
Injecting this bound into \eqref{super-ineq} and combining with \eqref{sub-ineq}, we deduce that if $(\bar u,\bar s)\in (\tilde \Omega \setminus \mathcal N^\a_\Gamma) \times ]0,T[$ and $(\bar x,\bar t)\in (\Omega\setminus \mathcal N^\a_\Gamma) \times ]0,T[$, then
\begin{align}
\eta 
&\leq K \pa{\e^\nu + \frac{\e}{\gamma}} + \tilde{P}(\bar{u}) - P(\bar{x}) \nonumber\\
&\leq K \pa{\e^\nu + \frac{\e}{\gamma}} + \Lip{P} |\bar{x}-\bar{u}| + \normL{P - \tilde{P}}_{L^\infty(\tilde{\Omega} \setminus \tilde{\Gamma})} \nonumber\\
&< 2K \pa{\e^\nu + \frac{\e}{\gamma}} + \Lip{P} K \gamma + \normL{P - \tilde{P}}_{L^\infty(\tilde{\Omega} \setminus \tilde{\Gamma})} \nonumber\\
&\leq K \pa{\e^\nu + \frac{\e}{\gamma} + \gamma}+\normL{P - \tilde{P}}_{L^\infty(\tilde{\Omega} \setminus \tilde{\Gamma})} \eqdef \bar{\eta} \label{eq:etabar} ,
\end{align}
for large enough constant $K > 0$, where we used \ref{assum:Gam} and \ref{assum:P} in the second inequality and  estimate \eqref{lip-pro-x-ct} in the third one. Then we conclude that for $\eta \geq \bar{\eta}$ either $(\bar{u},\bar{s}) \in \mathcal N^\a_\Gamma\times [0,T[\ \cup\  \tilde \Omega\times \{0\}$  or $(\bar{x},\bar{t}) \in \mathcal N^\a_\Gamma\times [0,T[\ \cup\  \Omega\times \{0\}$.\medskip

\item {\em Conclusion.} \label{step:proofconclusion}
We take $\eta \ge \bar\eta$. Assume first that $(\bar{x},\bar{t}) \in \mathcal N^\a_\Gamma\times [0,T[\ \cup\  \Omega\times \{0\}$. If $\bar t=0$, then 
\begin{align*}
\Psi_{{\gamma,\eta}}(\bar{u},\bar{s},\bar{x},\bar{t}) 
&\leq f^\e(\bar{u},\bar{s}) - \psi(\bar{x}) \\
&= (f^\e(\bar{u},\bar{s}) - f^\e(\bar{u},0)) + (\tilde{\psi}(\bar{u}) - \psi(\bar{u})) + (\psi(\bar{u}) - \psi(\bar{x})) \\
&\leq K \bar{s} + \normL{\psi-\tilde\psi}_{L^\infty(\tilde \Omega)} + \Lip{\psi} |\bar{x}-\bar{u}|\\
&\leq K(\eta+1) \gamma + \norm{\psi-\tilde\psi}_{L^\infty(\tilde \Omega)},
\end{align*}
where, in the second inequality, we used \eqref{lip-t-ct} in Theorem~\ref{lip-viscosity-J} to get the first term, and \ref{assum:Om} and \ref{assum:psi} to get the last two terms. In the last inequality, we invoked \eqref{lip-pro-x-ct} and \eqref{lip-pro-t-ct}. In the same way, if $\bar{x} \in \mathcal N_\Gamma^\a$ and $\bar t>0$, let $\tilde u \in \proj_{\tilde{\Gamma}}(\bar{x})$, i.e.,
\[
|\bar x-\tilde u|= d(\bar x,\tilde \Gamma) \le \distH(\Gamma, \tilde \Gamma)+\a .
\]
Such $\tilde u$ exists by closedness of $\tilde{\Gamma}$, see \ref{assum:Gam}. Since \eqref{eq:compatdomains} is in force under \ref{assum:compatdomains}, \eqref{eq:globlip-space-J} holds (see Theorem~\ref{lip-viscosity-J} and Lemma~\ref{lip-viscosity-J-1}). Using this with \ref{assum:psi} and \eqref{lip-pro-x-ct}, we obtain
\begin{equation}\label{eq:bndxonGamma}
\begin{aligned}
\Psi_{{\gamma,\eta}}(\bar{u},\bar{s},\bar{x},\bar{t}) 
&\leq f^\e(\bar{u},\bar{s}) - \psi(\bar{x}) \\
&= (f^\e(\bar{u},\bar{s}) - f^\e(\tilde{u},\bar{s})) + (\tilde{\psi}(\tilde{u}) - \psi(\tilde{u})) + (\psi(\tilde{u}) - \psi(\bar{x})) \\
&\leq K(|\bar{u}-\tilde{u}|+\e) + \normL{\psi-\tilde\psi}_{L^\infty(\tilde \Omega)} + \Lip{\psi} |\bar{x}-\tilde{u}| \\
&\leq K(|\bar{x}-\bar{u}|+\e) + \normL{\psi-\tilde\psi}_{L^\infty(\tilde \Omega)} + K|\bar{x}-\tilde{u}| + \Lip{\psi} |\bar{x}-\tilde{u}| \\
&\leq K (\gamma + \e) + \norm{\psi-\tilde\psi}_{L^\infty(\tilde \Omega)}+ K(\distH(\Gamma,\tilde \Gamma) +\a).
\end{aligned}
\end{equation}

We conclude that for all $(\bar{x},\bar{t}) \in \mathcal N^\a_\Gamma\times [0,T[\ \cup\  \Omega\times \{0\}$, and for $\eta \ge \bar{\eta}$, we have
\[
\Psi_{{\gamma,\eta}}(\bar{u},\bar{s},\bar{x},\bar{t})\leq K (\gamma + \e) + \norm{\psi-\tilde\psi}_{L^\infty(\tilde \Omega)}+ K ( \distH(\Gamma,\tilde \Gamma)+\a)+K \eta \gamma .
\]
The same bound holds for $(\bar{u},\bar{s}) \in \mathcal N^\a_\Gamma\times [0,T[\ \cup\  \tilde \Omega\times \{0\}$ whenever $\eta \geq \bar{\eta}$. Indeed, if $\bar{s}=0$ then
\begin{align*}
\Psi_{{\gamma,\eta}}(\bar{u},\bar{s},\bar{x},\bar{t}) 
&\leq \tilde{\psi}(\bar{u}) - f(\bar{x},\bar{t}) \\
&= (\tilde{\psi}(\bar{u}) - \psi(\bar{u})) + (\psi(\bar{u}) - \psi(\bar{x})) + (f(\bar{x},0) - f(\bar{x},\bar{t})) \\
&\leq \normL{\psi-\tilde\psi}_{L^\infty(\tilde \Omega)} + \Lip{\psi} |\bar{x}-\bar{u}| + K\bar{t} \\
&\leq K(\eta+1) \gamma + \normL{\psi-\tilde\psi}_{L^\infty(\tilde \Omega)},
\end{align*}
where we have now invoked \eqref{lip-t-ct-1} in Theorem~\ref{lip-viscosity}. If $\bar{u} \in \mathcal N^\a_\Gamma$ and $\bar{s} > 0$, define $\hat{x} \in \Gamma$ in the projection of $\bar u$ on $\Gamma$. Thus, using \eqref{eq:lipspace} in Theorem~\ref{lip-viscosity}, we arrive at
\begin{align*}
\Psi_{{\gamma,\eta}}(\bar{u},\bar{s},\bar{x},\bar{t}) 
&\leq \tilde{\psi}(\bar{u}) - f(\bar{x},\bar{t}) \\
&= (\tilde{\psi}(\bar{u}) - \psi(\bar{u})) + (\psi(\bar{u}) - \psi(\hat{x})) + (f(\hat{x},\bar{t}) - f(\bar{x},\bar{t})) \\
&\leq \normL{\psi-\tilde\psi}_{L^\infty(\tilde \Omega)} + \Lip{\psi} |\hat{x}-\bar{u}| + K |\hat{x}-\bar{x}| \\
&\leq  \normL{\psi-\tilde\psi}_{L^\infty(\tilde \Omega)} + \Lip{\psi} \a + K(\a+\gamma) \\
&\leq K (\alpha+ \gamma) + \norm{\psi-\tilde\psi}_{L^\infty(\tilde \Omega)}.
\end{align*}
%
%
Thus, taking $\eta = \bar{\eta}$ and $(u,s) \in \tilde{\Omega}_{T}$ we have from above that
\begin{align*}
f^\e(u,s) - f(u,s) - \bar{\eta}T \leq  &\Psi_{{\gamma,\eta}}(\bar{u},\bar{s},\bar{x},\bar{t}) \\
\leq & K(\gamma + \e) + \norm{\psi-\tilde\psi}_{L^\infty(\tilde \Omega)} + K (\distH(\Gamma,\tilde \Gamma)+\a)+ K \bar{\eta} \gamma.
\end{align*}
\\
Before concluding, we look at what happens when we revert the role of $f$ and $f^\e$. In this case, our reasoning remains valid with only a few changes. The main ingredient is to redefine $\Psi_{{\gamma,\eta}}$ as follows
\[
\Psi_{\gamma,\eta}(u,s,x,t) = f(x,t) - f^\e(u,s) - \frac{|x-u|^2}{2\gamma} -\frac{|t-s|^2}{2\gamma} - \eta t.
\]
Then all our bounds remain true, and with even  simpler arguments\footnote{This is the case for the analogous version of \eqref{ineq-2-ct} which will be derived by a simple triangle inequality (see also \eqref{ineq-2}). This asymmetry in the proofs when reverting the roles of $f^\e$ and $f$ is intriguing but not surprising.}. We leave the details to the reader for the sake of brevity.

Overall, we have shown that
\[ 
|f^\e(u,s) - f(u,s)| \leq  K(\gamma + \e) + \norm{\psi-\tilde\psi}_{L^\infty(\tilde \Omega)} + K( \distH(\Gamma,\tilde \Gamma)+\a)+ \bar{\eta} (\gamma+T).
\]
With the optimal choice $\gamma = \e^{1/2}$, taking the supremum over $(u,s)$ and after rearrangement, we get
\begin{align*}
\normL{f^\e-f}_{L^\infty(\tilde{\Omega} \times  [0,T[)} \leq&  K\pa{(T+1)\e^{\min(\nu,1/2)} + \e} + K (T+\e^{1/2})\normL{P - \tilde{P}}_{L^\infty(\tilde{\Omega} \setminus \tilde{\Gamma})} \\
&+ \normL{\psi - \tilde{\psi}}_{L^\infty(\tilde{\Omega})}+ K( \distH(\Gamma,\tilde \Gamma)+\a),
\end{align*}
which is the claimed bound since $\e \in ]0,1]$ and $\a=\e^{1/2}$.
\end{enumerate}
\end{proof}

\subsection{Forward Euler discrete time non-local to local error bound}\label{subsec:eikconvdiscrete}
We consider the time-discrete approximation of \eqref{cauchy-J} using forward Euler discretization. Then we will show an error estimate between a solution of this equation with the continuous viscosity solution of \eqref{cauchy}.

Using the forward/explicit Euler discretization scheme, a time-discrete counterpart of \eqref{cauchy-J} reads
\begin{equation}\tag{\textrm{$\mathcal{P}_{\e}^{\rm{FD}}$}}\label{cauchy-J-discrete-fw}
\begin{cases}
\frac{f^\e(u,t) - f^\e(u,t-\Delta t)}{\Delta t} = - \abs{ \nabla_{J_\e}^- f^\e(u,t-\Delta t)}_\infty + \tilde{P}(u),  & (u,t) \in   (\tilde{\Omega} \setminus \tilde{\Gamma}) \times \set{t_1,\ldots,t_{N_T}} ,\\
f^\e(u,t) = \tilde{\psi} (u), & (u,t) \in \partial\tilde{\Omega}_{N_T} ,
\end{cases}
\end{equation}
where $t_{i} = i \Delta t$ for all $i \in \set{0,\ldots,N_T}$. 

In Appendix~\ref{sec:cauchy-J-discrete}, we prove that \eqref{cauchy-J-discrete-fw} is well-posed. Indeed, Lemma~\ref{lem-existence-lip-fw} shows existence and regularity of a discrete-time solution (in the sense of Definition~\ref{def:discretesolution}). Uniqueness follows from the comparison principle in Lemma~\ref{lem-comparison}.\\


We are now in position to state the following error estimate.

\begin{thm}\label{thm:discrete-fw-estimate}
Let $T>0$, $\e_0=\min(1/ {(2r_g)^2},1)$ and $\e \in ]0,\e_0]$. Suppose that assumptions~\ref{assum:Om}--\ref{assum:compatdomains} hold, and that  $\distH(\Gamma,\tilde \Gamma) = O(\e^{\min(\nu,1/2)})$. Let $f$ be the unique viscosity solution of \eqref{cauchy} and $f^\e$ be a solution of \eqref{cauchy-J-discrete-fw}. Assume also that
\begin{eqnarray}\label{cond-delta-t}
0 < \Delta t \leq \frac{\e C_g}{\sup_{t \in \R_+} g(t)}.
\end{eqnarray} 
Then there exists a constant $K > 0$ depending only on the dimension $m$, $\norm{\psi}_{L^\infty(\Omega)}$, $\norm{P}_{L^\infty(\Omega\setminus\Gamma)}$, $\Lip{\psi}$, $\Lip{\tilde{\psi}}$, $\Lip{P}$, $\Lip{\tilde{P}}$, $L_{g}$, $C_g$ and $c_g$ such that for any $\e$ small enough
\begin{align*}
\normL{f^\e-f}_{L^\infty\pa{\tilde{\Omega} \times \set{0,\ldots,t_{N_T}}}}  
\leq& K(T+1)\pa{\e^{\min(\nu,1/2)} + \Delta t^{1/2} + \frac{\Delta t}{\e}
+ \normL{P - \tilde{P}}_{L^\infty(\tilde{\Omega}\setminus\tilde{\Gamma})}} \\
&+ \normL{\psi - \tilde{\psi}}_{L^\infty(\tilde{\Omega})}.
\end{align*}
In particular, if $\tilde{P}=P$ on $\tilde{\Omega}\setminus\tilde{\Gamma}$ and $\tilde{\psi}=\psi$ on $\tilde{\Omega}$, then for $\Delta t=o(\e)$, we have
\begin{align*}
\lim_{\e \to 0, \Delta t \to 0} \normL{f^\e-f}_{L^\infty\pa{\tilde{\Omega} \times \set{0,\ldots,t_{N_T}}}}  = 0 .
\end{align*}
The fastest convergence rate in $\e$ is then achieved when $\Delta t = O(\e^{3/2})$ and $\nu=1/2$ provided that $\distH(\Omega,\tilde \Omega) = O(\e^{3/2})$ and $\distH(\Gamma,\tilde \Gamma) = O(\e^{1/2})$. 
\end{thm}

\begin{rem}
It is worth noting that \eqref{cond-delta-t} is a CFL condition; \tcb{see \cite{de2013courant} for a recent overview.} Obviously, we have $\sup_{t \in \R_+} g(t) < +\infty$ under our assumptions since $g$ is a continuous function on its compact support. It is also possible to discretize \eqref{cauchy-J}  in time using a backward/implicit Euler scheme:
\begin{equation}\tag{\textrm{$\mathcal{P}_{\e}^{\rm{BD}}$}}\label{cauchy-J-discrete-bw}
\begin{cases}
\frac{f^\e(u,t) - f^\e(u,t-\Delta t)}{\Delta t} = - \abs{ \nabla_{J_\e}^- f^\e(u,t)}_\infty + \tilde{P}(u),  & (u,t) \in   (\tilde{\Omega} \setminus \tilde{\Gamma}) \times \set{t_1,\ldots,t_{N_T}} ,\\
f^\e(u,t) = \tilde{\psi} (u), & (u,t) \in \partial\tilde{\Omega}_{N_T} ,
\end{cases}
\end{equation}
In that case, the CFL condition \eqref{cond-delta-t} is not required anymore and we can prove the following error estimate
\[
\normL{f^\e-f}_{L^\infty\pa{\tilde{\Omega} \times \set{0,\ldots,t_{N_T}}}}  
\leq K(T+1)\pa{\e^{\min(\nu,1/2)} + \Delta t^{1/2}
+ \normL{P - \tilde{P}}_{L^\infty(\tilde{\Omega}\setminus\tilde{\Gamma})} }
+ \normL{\psi - \tilde{\psi}}_{L^\infty(\tilde{\Omega})}.
\]
The proof is completely similar to the one of Theorem \ref{thm:discrete-fw-estimate}, the only difference is that $\bar{\eta}$ in the proof will be independent of the term $\Delta t / \e$. It is worth noting that in the backward Euler scheme, the difficulty is to construct a discrete solution. This is done for the reader's convenience in Lemma~\ref{discrete-existence} in the appendix.
\end{rem}
The proof of Theorem  \ref{thm:discrete-fw-estimate} is quite similar to the one of Theorem~\ref{thm:continuous-time-estimate}. We highlight only the steps where we have to handle properly the discrete time approximation. For instance, we will need the Lipschitz regularity properties of $f^\e$ both in time and space (see Lemma~\ref{lem-existence-lip-fw}).

\begin{proof}
Again, $K$ will denote in this proof any positive constant that depends only on the data, but may change from one line to another. Here we focus on the case $f-f^\e$ on purpose to complement the details provided in the proof Theorem~\ref{thm:continuous-time-estimate}. 

\begin{enumerate}[label=Step~\arabic*.]
\item {\em Test-function and maximum point.} 

For $\gamma > 0$ and $\eta > 0$, we consider maximizing over $\Omega_T \times \tilde{\Omega}_{N_T}$ the test-function
\begin{eqnarray*}\label{varphi}
\Psi_{\gamma,\eta}(x,t,u,s)= f(x,t) - f^\e(u,s) - \frac{|x-u|^2}{2\gamma} -\frac{|t-s|^2}{2\gamma} - \eta t.
\end{eqnarray*}

Since $\Omega_T \times \tilde{\Omega}_{N_T}$ is compact and $\Psi_{\gamma,\eta}$ is continuous, the maximum is attained at some point $(\bar{x},\bar{t}, \bar{u},\bar{t}_i)$. Exactly as in the proof of Theorem~\ref{thm:continuous-time-estimate}, we have
\begin{align}
|\bar{x}-\bar{u}| &\leq K \gamma \label{lip-pro-x} \qandq \\
|\bar{t} - \bar{t}_i| &\leq (K+\eta) \gamma. \label{lip-pro-t}
\end{align}

\item {\em Excluding interior points from the maximum.}

\noindent We show that for $\eta$ large enough, we have either $(\bar{x},\bar{t}) \in \partial \Omega_T$ or $(\bar{u},\bar{t}_i) \in \partial \tilde{\Omega}_{N_T}$. We argue again by contradiction and assume that $(\bar{x},\bar{t}) \in  \Omega\setminus \Gamma \times ]0,T[$  and $(\bar{u},\bar{t}_i) \in   \tilde{\Omega}\setminus \tilde \Gamma \times \{t_1,\dots, t_{N_T}\}$. Using that $(\bar{x},\bar{t})$ is a maximum point of the function $(x,t) \mapsto \Psi_{\gamma,\eta}(x,t, \bar{u},\bar{t}_i) $ and the fact that $f$ is a  viscosity sub-solution of \eqref{cauchy}, we have
\begin{eqnarray}\label{1st-ineq}
\eta + \frac{\bar{t} - \bar{t}_i}{\gamma} \leq - \frac{|\bar{x}-\bar{u}|}{\gamma} + P(\bar{x}).
\end{eqnarray}

Using now that $\bar{t}_i > 0$ and that $f^\e$ is a solution of \eqref{cauchy-J-discrete-fw}, we have
\begin{equation}\label{2nd-ineq}
\frac{f^\e(\bar{u},\bar{t}_i) - f^\e(\bar{u},\bar{t}_i-\Delta t)}{\Delta t} = - \babs{\nabla_{J_\e}^- f^\e(\bar{u},\bar{t}_i-\Delta t)}_\infty + \tilde{P}(\bar{u}).
\end{equation}
We set $\varphi: (u,s) \in \tilde{\Omega}_{N_T} \mapsto f(\bar x,\bar t) - \frac{|\bar{x}-u|^2}{2\gamma} -\frac{|\bar{t}-s|^2}{2\gamma} - \eta \bar{t}$. In particular, $(\bar{u},\bar{t}_i)$ is the minimum point of $f^\e-\varphi$ over $\tilde{\Omega}_{N_T}$. This implies that

\begin{eqnarray*}
f^\e(\bar{u},\bar{t}_i) - f^\e(\bar{u},\bar{t}_i- \Delta t) \leq \varphi (\bar{u},\bar{t}_i) -\varphi(\bar{u},\bar{t}_i - \Delta t),
\end{eqnarray*}

and so
\begin{eqnarray}\label{ineq-1}
\frac{f^\e(\bar{u},\bar{t}_i) - f^\e(\bar{u},\bar{t}_i- \Delta t) }{\Delta t} \leq \frac{\bar{t} - \bar{t}_i}{\gamma} + \frac{\Delta t}{2 \gamma}.
\end{eqnarray}
$(\bar{x},\bar{t},\bar{u},\bar{t}_i)$ is a maximizer of $\Psi_{\gamma,\eta}$, whence we get 
\[
f^\e(\bar{u},\bar{t}_i) - f^\e(v,\bar{t}_i) \leq \frac {|\bar{x} - v|^2 - |\bar{x} - \bar{u}|^2}{2\g}, \qquad \forall v \in \tilde{\Omega} .
\]
Thus we estimate the right hand side of \eqref{2nd-ineq} to show that
\begin{equation}\label{ineq-2}
\begin{aligned}
\babs{\nabla_{J_\e}^- f^\e(\bar{u},\bar{t}_i-\Delta t)}_\infty
&=  \max_{v\in\tilde{\Omega},|\bar{u}-v| \in \e\Sg}J_\e(\bar{u},v) (f^\e(\bar{u},\bar{t}_i-\Delta t) -f^\e(v,\bar{t}_i-\Delta t)) \\
&= \max_{v\in\tilde{\Omega},|\bar{u}-v| \in \e\Sg}J_\e(\bar{u},v) (f^\e(\bar{u},\bar{t}_i-\Delta t) - f^\e(\bar{u},\bar{t}_i) + f^\e(\bar{u},\bar{t}_i) - f^\e(v,\bar{t}_i) \\
&+ f^\e(v,\bar{t}_i) - f^\e(v,\bar{t}_i-\Delta t))\\
&\leq \max_{v\in\tilde{\Omega},|\bar{u}-v| \in \e\Sg} J_\e(\bar{u},v) (K \Delta t + f^\e(\bar{u},\bar{t}_i) -f^\e(v,\bar{t}_i)) \\
&\leq \max_{v\in\tilde{\Omega},|\bar{u}-v| \in \e\Sg} J_\e(\bar{u},v)\pa{ K \Delta t +(2\gamma)^{-1}\left( |\bar{x}-v|^2 - |\bar{x}-\bar{u}|^2} \right) \\
&\leq K\Delta t \max_{v\in\tilde{\Omega},|\bar{u}-v| \in \e\Sg} J_\e(\bar{u},v) \\
& + (2\gamma)^{-1}\max_{v\in\tilde{\Omega},|\bar{u}-v| \in \e\Sg} J_\e(\bar{u},v)\bpa{|\bar{x}-v| - |\bar{x}-\bar{u}|}\bpa{|\bar{x}-v| - |\bar{x}-\bar{u}| + 2 |\bar{x}-\bar{u}|}  \\
&\leq K \frac{ \Delta t}\e \sup_{t \in \R_+}g(t)
+ (2\gamma)^{-1}\max_{v\in\tilde{\Omega},|\bar{u}-v| \in \e\Sg} J_\e(\bar{u},v){|\bar{u}-v|}\bpa{|\bar{u}-v| + 2 |\bar{x}-\bar{u}|}  \\
&\leq  K\frac{\Delta t}{\e} + \max_{v\in\tilde{\Omega},|\bar{u}-v| \in \e\Sg} \frac{|\bar{x}-\bar{u}|}{\gamma} \frac{|\bar{u}-v|}{C_g \e} g\pa{\frac{|\bar{u}-v|}{\e}} \\
&+ \max_{v\in\tilde{\Omega},|\bar{u}-v| \in \e\Sg} \frac{|\bar{u}-v|^2}{2\gamma C_g \e} g\pa{\frac{|\bar{u}-v|}{\e}} \\
&\leq  K\frac{\Delta t}{\e} + \frac{|\bar{x}-\bar{u}|}{\gamma} + r_g\frac{\e}{2\gamma} .
\end{aligned}
\end{equation}


Plugging \eqref{ineq-1} and \eqref{ineq-2} into \eqref{2nd-ineq} we get

\begin{eqnarray}\label{3rd-ineq}
\frac{\bar{t} - \bar{t}_i}{\gamma} + \frac{\Delta t}{2 \gamma} \geq  - K\frac{\Delta t}{\e} -  \frac{|\bar{x}-\bar{u}|}{\gamma} - K\frac{\e}{\gamma} + \tilde{P}(\bar{u}).
\end{eqnarray}

From \eqref{1st-ineq} and \eqref{3rd-ineq}, we finally obtain

\begin{eqnarray*}
\eta &\leq& K \pa{\frac{\Delta t + \e}{ \gamma}  + \frac{\Delta t}{\e}} + P(\bar{x}) - \tilde{P}(\bar{u})\\
&\leq&   K \pa{\frac{\Delta t + \e}{ \gamma}  + \frac{\Delta t}{\e}} + K |\bar{x}-\bar{u}| + \normL{P - \tilde{P}}_{L^\infty(\tilde{\Omega} \setminus \tilde{\Gamma})} \\
&<& K \pa{\frac{\Delta t + \e}{ \gamma}  + \gamma + \frac{\Delta t}{\e}}+\normL{P - \tilde{P}}_{L^\infty(\tilde{\Omega} \setminus \tilde{\Gamma})} \eqdef \bar{\eta}.
\end{eqnarray*}
We then conclude that either $(\bar{x},\bar{t}) \in \partial \Omega_T$ or $(\bar{u},\bar{t}_i) \in \partial \tilde{\Omega}_{N_T}$ for $\eta\ge \bar \eta$. When reverting the roles of $f^\e$ and $f$, only $\bar{\eta}$ will be changed taking the additional terms $\e^\nu$ and $\a$ (see the proof of Theorem~\ref{thm:continuous-time-estimate}). The rest of the proof is exactly the same as \ref{step:proofconclusion} in the proof of Theorem~\ref{subsec:eikconvcont}, where we now invoke Lemma~\ref{lem-existence-lip-fw}.
\end{enumerate}
\end{proof}


\section{Application to graph sequences}\label{sec:eikconvgraphs}

Let $G_n=(V_n,w_n)$ be a finite weighted graph with non-negative edge weights $w_n$. Here $V_n$ is the set of $n$ vertices/nodes $\set{u_1,\ldots,u_n} \subset \Omega$, $E_n \subset V_n^2$ is the set of edges, and the weights $w_n$ are given by the kernel $J$ at scale $\e_n$, i.e., $w_n(u_i,v_j) = J_{\e_n}(u_i,v_j)$.

Let $\Gamma_n \subset V_n$. For a time interval $[0,T[$ and $N_T \in \N$, \tcb{we use the shorthand notation $(V_n\setminus \Gamma_n)_{N_T}=(V_n \setminus \Gamma_n) \times \set{t_1,\ldots,t_{N_T}}$ and $\partial(V_n)_{N_T}=(\Gamma_n \times \set{t_1,\ldots,t_{N_T}}) \cup V_n \times \set{0} $.} We now consider the fully discretized Eikonal equation on $G_n$ with a forward Euler time-discretization as
\begin{equation}\tag{\textrm{$\mathcal{P}_{G_n}^{\rm{FD}}$}}\label{cauchy-graph-fw}
\begin{cases}
\frac{f^{n}(u,t) - f^{n}(u,t-\Delta t)}{\Delta t} = - \abs{\nabla_{w_n}^- f^{n}(u,t-\Delta t)}_\infty + \tilde{P}(u),  &\tcb{ (u,t) \in  (V_n\setminus \Gamma_n)_{N_T} } ,\\
f^{n}(u,t) = \tilde{\psi}(u), & \tcb{ (u,t) \in \partial(V_n)_{N_T} },
\end{cases}
\end{equation}
where $t_{i} = i \Delta t$ for all $i \in \set{0,\ldots,N_T}$.

In the notation of \eqref{cauchy-J-discrete-fw}, it is easy to identify $V_n$ with $\tilde{\Omega}$ and $\Gamma_n$ with $\tilde{\Gamma}$. Our aim in this section is to establish consistency of solutions to \eqref{cauchy-graph-fw} as $n \to +\infty$ and $\Delta t \to 0$. \\

In practice, we do not have that much control over the way the vertices $V_n$ in the graph are constructed; the precise configuration of points may not be known, or the points can be obtained by sampling through an acquisition device (e.g., point clouds), or given from a learning or modeling process (e.g., images). It then appears more realistic to consider graphs $G_n$ on random point configurations $V_n$, and then conveniently estimate the probability of achieving a prescribed level of consistency as a function of $n$. 

Towards this goal, we will consider a random graph model whose nodes are latent random variables independently and identically sampled on $\Omega$. This random graph model is inspired from \cite{Bollobas07} and is quite standard. More precisely, we construct $V_n$ and the boundary $\Gamma_n$ as follows:
\begin{defn}\label{def:randgraph}
Given a probability measure $\mu$ over $\Omega$ and $\e_n > 0$:
\begin{enumerate}
\item draw the vertices in $V_n$ as a sequence of independent and identically distributed variables $\pa{u_i}_{i=1}^n$ taking values in $\Omega$ and whose common distribution is $\mu$;
\item set $\Gamma_n = \set{u_i \in V_n:~ d(u_i,\Gamma) \leq a\e_n^{1+\nu}/(2\sqrt{m})}$, $\nu > 0$.
\end{enumerate}
\end{defn}
From now on, we assume that \\

\noindent\fbox{\parbox{0.975\textwidth}{
\begin{enumerate}[label=({\textbf{H.\arabic*}}),itemindent=5ex,start=13]
\item $\mu$ has a density $\rho$ on $\Omega$ with respect to the volume measure, and $\inf_{\Omega} \rho > 0$. \label{assum:densitylbd}
\end{enumerate}}}
{~}\\

A typical example is that of the uniform probability distribution on $\Omega$, in which case $\rho(u)=(\int_\Omega \dvol(x))^{-1}$ for $u \in \Omega$, where $\dvol$ is the volume measure. Though we will focus on this setting, our results can be extended following the developments hereafter to other sampling models, in particular those adapted to the manifold geometry, in which case the covering arguments that we will use will be done with geodesic balls. We will not elaborate more on this in this paper. We observe in passing that by construction, $V_n$ and $\Gamma_n$ are compact sets, and that $V_n \setminus \Gamma_n \subset \Omega \setminus \Gamma$.\\ 
Before stating the main result of this section, the following lemma gives a proper choice of $\e_n$ for which the construction of Definition~\ref{def:randgraph} ensures that the key assumption \ref{assum:compatdomains} is in force together with  $\Gamma_n \neq \emptyset$ and $\distH(\Gamma,\Gamma_n) = O(\e_n^{1+\nu})$ with high probability. To lighten notation, we define the event
\begin{equation}\label{eq:eventEn}
\En = \set{\text{\ref{assum:compatdomains} holds} \qandq \distH(\Gamma,\Gamma_n) \leq a\e_n^{1+\nu}/(2\sqrt{m})} .
\end{equation}

\begin{lem}\label{lem:compatassum-graph}
Let $V_n$ and $\Gamma_n$ generated according to Definition~\ref{def:randgraph} where $\mu$ satisfies \ref{assum:densitylbd}. Then, there exists two constants $K_1 > 0$ and $K_2 > 0$ that depend only on $m$, $a$ and $\diam(\Omega)$, and for any $\tau > 0$ there exists $n(\tau) \in \N$ such that for $n \geq n(\tau)$, taking

\begin{equation}\label{eq:epsn}
\e_n^{1+\nu} = K_1(1+\tau)^{1/m}\pa{\frac{\log n}{n}}^{1/m} ,
\end{equation}
the event $\En$ in \eqref{eq:eventEn} holds with probability at least $1 - K_2 n^{-\tau}$.
\end{lem}
See Appendix~\ref{sec:prooflemcompatassum-graph} for the proof.\\


\noindent
We are now ready to establish a quantified version of uniform convergence in probability of $f^n$ towards $f$.

\begin{thm}\label{thm:graph-bw-estimate}
Let $T, \nu >0$, and $V_n$ and $\Gamma_n$ be constructed according to Definition~\ref{def:randgraph} where $\mu$ satisfies \ref{assum:densitylbd}. Suppose that assumptions~\ref{assum:Om}--\ref{assum:psisssol-J} hold\footnote{It is clear that our assumptions \ref{assum:Om}--\ref{assum:Gam} concern only $\Omega$ and $\Gamma$ and not $V_n$ and $\Gamma_n$ which comply with \ref{assum:Om}--\ref{assum:Gam} by construction.}. Let $f$ be the unique viscosity solution of \eqref{cauchy} and $f^{n}$ be a solution of \eqref{cauchy-graph-fw}. Take $\Delta t = o(\e_n)$ where $\e_n$ is as given in \eqref{eq:epsn}. Then, the following holds.
\begin{enumerate}[label=(\roman*)]
\item There exists two constants $K > 0$ and $K_2 > 0$ that depend only on $m$, $a$, $\diam(\Omega)$, $\norm{\psi}_{L^\infty(\Omega)}$, $\norm{P}_{L^\infty(\Omega\setminus\Gamma)}$, $\Lip{\psi}$, $\Lip{\tilde{\psi}}$, $\Lip{P}$, $\Lip{\tilde{P}}$, $c_g$, $C_g$, $L_g$ and $\nu$, and for any $\tau > 0$, there exists $n(\tau) \in \N$ such that for $n \geq n(\tau)$,
\begin{align}\label{eq:bndrandom}
\normL{f^n-f}_{L^\infty\pa{V_n \times \set{0,\ldots,t_{N_T}}}}  
&\leq K(T+1)\Bigg((1+\tau)^{\frac{\min(\nu,1/2)}{(1+\nu)m}}\pa{\frac{\log n}{n}}^{\frac{\min(\nu,1/2)}{(1+\nu)m}} \\
&\qquad\qquad\qquad+ \pa{1+(1+\tau)^{\frac{1}{2(1+\nu)m}}\pa{\frac{\log n}{n}}^{\frac{1}{2(1+\nu)m}}}o(1) \Bigg) \nonumber\\
&\qquad\qquad\qquad+ K(T+1)\normL{P - \tilde{P}}_{L^\infty(V_n\setminus\Gamma_n)} 
+ \normL{\psi - \tilde{\psi}}_{L^\infty(V_n)} \nonumber.
\end{align}
with probability at least $1-K_2 n^{-\tau}$. The best convergence rate is $O\pa{\frac{\log n}{n}}^{\frac{1}{3m}}$ obtained for $\nu=1/2$ and $\Delta t=O(\e_n^{3/2})$.
\item Let $\delta_n(\tau)$ be the right hand side of \eqref{eq:bndrandom}. If $\tau > 1$, then 
\[
\Pr\pa{\text{$\normL{f^n-f}_{L^\infty\pa{V_n \times \set{0,\ldots,t_{N_T}}}} > \delta_n(\tau)$ infinitely often}} = 0 .
\]
\item Take $\Delta t=O(\e_n^{3/2})$. Assume that $\tilde{P}=P$ on $V_n\setminus\Gamma_n$ and $\tilde{\psi}=\psi$ on $\Omega_n$, then
\begin{align*}
\lim_{n \to +\infty} \normL{f^n-f}_{L^\infty\pa{V_n \times \set{0,\ldots,t_{N_T}}}}  = 0 \quad \text{almost surely}.
\end{align*}
\end{enumerate}
\end{thm}

\begin{proof}
\begin{enumerate}[label=(\roman*)]
\item To get the error bound \eqref{eq:bndrandom}, combine Theorem~\ref{thm:discrete-fw-estimate} and Lemma~\ref{lem:compatassum-graph} and observe that 
\[
\distH(\Gamma,\Gamma_n) \leq a\e_n^{1+\nu}/(2\sqrt{m}) = o\pa{\e_n^{\min(\nu,1/2)}}
\]
with the stated probability. 
 
\item We have
\[
\sum_{n \geq n(\tau)}\Pr\pa{\normL{f^n-f}_{L^\infty\pa{V_n \times \set{0,\ldots,t_{N_T}}}} > \delta_n} < K_2\sum_{n \geq n(\tau)} n^{-\tau} < +\infty ,
\]
since $\tau > 1$. The claim then follows from the (first) Borel-Cantelli lemma.

\item In this case, we have $\delta_n(\tau)=K(T+1)(1+\tau)^{\frac{1}{3m}}\pa{\frac{\log n}{n}}^{\frac{1}{3m}}$ for a possibly higher constant $K$. Define the event
\[
\An = \set{\normL{f^n-f}_{L^\infty\pa{V_n \times \set{0,\ldots,t_{N_T}}}} \leq \delta_n(\tau)} .
\]
For any $\delta > 0$, $q > 1$ and $n$ large enough, we have by the Tchebychev inequality that
\begin{align*}
&\Pr\pa{\normL{f^n-f}_{L^\infty\pa{V_n \times \set{0,\ldots,t_{N_T}}}} > \delta}
\leq \delta^{-3mq}\E\pa{\normL{f^n-f}_{L^\infty\pa{V_n \times \set{0,\ldots,t_{N_T}}}}^{3mq}} \\
&= \delta^{-3mq}\pa{\E\pa{\normL{f^n-f}_{L^\infty\pa{V_n \times \set{0,\ldots,t_{N_T}}}}^{3mq}\mathbf{1}_{\An}} 
+ \E\pa{\normL{f^n-f}_{L^\infty\pa{V_n \times \set{0,\ldots,t_{N_T}}}}^{3mq}\mathbf{1}_{\An^c}}} \\
&\leq \delta^{-3mq}\pa{\E\pa{\normL{f^n-f}_{L^\infty\pa{V_n \times \set{0,\ldots,t_{N_T}}}}^{3mq}\mathbf{1}_{\An}} 
+ \E\pa{\normL{f^n-f}_{L^\infty\pa{V_n \times \set{0,\ldots,t_{N_T}}}}^{3mq}\mathbf{1}_{\An^c}}} \\
&\leq \delta^{-3mq}\pa{\delta_n(\tau)^{3mq} + C^{3mq} \Pr\pa{\An^c}} \\
&\leq \delta^{-3mq}\pa{\delta_n(\tau)^{3mq} + K_2 C^{3mq} n^{-\tau}} \\
&\leq \delta^{-3mq} K_{\tau,T,m,q} \pa{\pa{\frac{\log n}{n}}^{q} + n^{-\tau}} ,
\end{align*}
for some constant $K_{\tau,T,m,q} > 0$, and where we used the fact that $\normL{f^n-f}_{L^\infty\pa{V_n \times \set{0,\ldots,t_{N_T}}}}$ is almost surely bounded by some constant $C > 0$. Since $q > 1$, the right-hand side is summable for any $\tau > 1$. The claim then follows using again the (first) Borel-Cantelli lemma.
\end{enumerate}
\end{proof}

\begin{rem}\label{rem:graph-bw-estimateexpect}
One can also easily derive from \eqref{eq:bndrandom} a bound in expectation. Indeed arguing as in the proof of the third claim of Theorem~\ref{thm:graph-bw-estimate}, we have
\begin{align*}
&\E\pa{\normL{f^n-f}_{L^\infty\pa{V_n \times \set{0,\ldots,t_{N_T}}}}} \\
\leq 
&K(T+1)\Bigg((1+\tau)^{\frac{\min(\nu,1/2)}{(1+\nu)m}}\pa{\frac{\log n}{n}}^{\frac{\min(\nu,1/2)}{(1+\nu)m}} 
+ \pa{1+(1+\tau)^{\frac{1}{2(1+\nu)m}}\pa{\frac{\log n}{n}}^{\frac{1}{2(1+\nu)m}}}o(1) \Bigg) \\
&+ T\normL{P - \tilde{P}}_{L^\infty(V_n\setminus\Gamma_n)} + \normL{\psi - \tilde{\psi}}_{L^\infty(V_n)} 
+ CK_2n^{-\tau},
\end{align*}
When $\tilde{P}=P$ and $\tilde{\psi}=\psi$ on $V_n\setminus\Gamma_n$ and $\Omega_n$ respectively, we again conclude that $\normL{f^n-f}_{L^\infty\pa{V_n \times \set{0,\ldots,t_{N_T}}}}$ converges to $0$ in expectation as $n \to +\infty$.
\end{rem}

\appendix 

\section{Smoothness of the distance function} 
\label{sec:distsmooth}


The regularity of the distance function to a set from that of the set itself is a classical and well understood subject. Indeed, characterizing some classes of $C^p$-smooth submanifolds of an arbitrary Hilbert space via some smoothness properties of square distance functions (or projection mappings) has been studied by many authors \cite{Poly84,Canino88,Shapiro94,Thibault19}; see also the survey \cite{Thibault10}. On $\R^m$, such results can be found in \cite{Ambrosio96}. We summarise this in the following proposition.
%

\begin{prop}
Let $p \geq 1$ be an integer and $\Gamma \in \R^m$ be a compact $C^{p+1}$-smooth submanifold without boundary. Then there is $a_0 > 0$ such that $d(\cdot,\Gamma)$ is $C^p$ on $\NeigGama \setminus \Gamma$ and $|\nabla d(x,\Gamma)|=1$ for all $x \in \NeigGama \setminus \Gamma$.
\end{prop}


\section{Well-posedness and regularity properties of \eqref{cauchy-J-discrete-fw}} 
\label{sec:cauchy-J-discrete}

We first define the notions of discrete sub- and super-solution.
\begin{defn}[Discrete sub- and super-solution]\label{def:discretesolution}
We say that $f^\e$ is a sub-solution of \eqref{cauchy-J-discrete-fw} if for all $(u,t) \in   (\tilde \Omega \setminus\tilde  \Gamma) \times \set{t_1,\ldots,t_{N_T}}$
\[
\frac{f^\e(u,t) - f^{\e}(u,t-\Delta t)}{\Delta t} \le - \abs{\nabla_{J_\e}^- f^{\e}(u,t-\Delta t)}_\infty + \tilde{P}(u) ,
\]
and if for all $(u,t) \in \partial \tilde \Omega_{N_T}$,
\[
f^{\e}(u,t) \le \tilde{\psi}(u) .
\]
In the same way, we say that $f^\e$ is a super-solution of \eqref{cauchy-J-discrete-fw} if for all $(u,t) \in  (\tilde \Omega \setminus\tilde  \Gamma) \times \set{t_1,\ldots,t_{N_T}}$
\[
\frac{f^\e(u,t) - f^{\e}(u,t-\Delta t)}{\Delta t} \ge  - \abs{\nabla_{J_\e}^- f^{\e}(u,t-\Delta t)}_\infty + \tilde{P}(u) ,
\]
and if for all $(u,t) \in \partial \tilde \Omega_{N_T}$,
\[
f^{\e}(u,t) \ge \tilde{\psi}(u).
\]
$f^\e$ is a discrete solution of \eqref{cauchy-J-discrete-fw} if it is both a discrete sub-solution and super-solution.
\end{defn}

We start with a comparison principle, which is a direct consequence of the monotonicity.
\begin{lem}[Comparison principle for the scheme \eqref{cauchy-J-discrete-fw}]\label{lem-comparison}
Assume that \ref{assum:Om}, \ref{assum:Gam} and \ref{assum:gpos} hold, and that $f^\e,g^\e$ are respectively bounded sub- and super-solution of \eqref{cauchy-J-discrete-fw}. Assume also that the CFL condition \eqref{cond-delta-t} holds. Then
\begin{eqnarray}
\sup_{\tilde{\Omega} \times  \set{0,\ldots,t_{N_T}}}\pa{f^\e - g^\e} \leq \sup_{ \tilde{\Gamma} \times  \set{t_1,\ldots,t_{N_T}} \cup \tilde{\Omega} \times \set{0}} |f^\e - g^\e |.
\end{eqnarray}
\end{lem}

\begin{proof}
Since the scheme is invariant by addition of constant, we can assume that $f^\e \leq g^\e$ on $\tilde{\Gamma} \times  \{t_1,\ldots,t_{N_T} \} \cup \tilde{\Omega} \times \{ 0 \}$, and prove that $f^\e \leq g^\e$ on $\tilde{\Omega} \times  \{0,\ldots,t_{N_T}\}$. \medskip

We argue by contradiction, and suppose that for $\eta > 0$ small enough, we have
\begin{eqnarray}
M_\eta = \sup_{(u,t) \in \tilde{\Omega} \times  \set{0,\ldots,t_{N_T}}}f^\e(u,t) - g^\e(u,t) - \eta t > 0. 
\end{eqnarray}

By upper semi-continuity of the objective and compactness of $\tilde \Omega \times \{0,\dots,t_{N_T}\}$, the supremum is actually a maximum achieved at some point $(\bar{u},\bar{t})$. Since  $f^\e \leq g^\e$ on $\tilde{\Gamma} \times \{t_1,\ldots,t_{N_T}\} \cup \tilde{\Omega} \times \{ 0 \}$ and $M_\eta > 0$ for $\eta$ small enough,  we deduce that  $(\bar{u},\bar{t}) \in (\tilde{\Omega} \setminus \tilde{\Gamma}) \times \{t_1,\ldots,t_{N_T} \}$. At the maximum point, we have 
\[
f^\e(\bar{u},\bar{t}) - g^\e(\bar{u},\bar{t}) - \eta \bar{t} \geq f^\e(\bar{u},\bar{t}-\Delta t) - g^\e(\bar{u},\bar{t}-\Delta t) - \eta (\bar{t} - \Delta t)
\]
and 
\[f^\e(\bar{u},\bar{t}) - g^\e(\bar{u},\bar{t}) - \eta \bar{t} \geq f^\e(y,\bar{t}-\Delta t) - g^\e(y,\bar{t}-\Delta t) - \eta (\bar{t} - \Delta t)
\]
Moreover, using that $f^\e,g^\e$ are respectively sub- and super-solution of \eqref{cauchy-J-discrete-fw} and remarking that \eqref{cond-delta-t} implies in particular that $\displaystyle \Delta t \leq \min_{y \in \tilde{\O}}\frac{1}{J_\e(\bar{u},y)}$, we get
\begin{eqnarray*}
0 
& \geq& \frac{f^\e(\bar{u},\bar{t}) - f^\e(\bar{u},\bar{t}-\Delta t)}{\Delta t} + \max_{y \in \tilde{\Omega}}J_\e(\bar{u},y)  (f^\e(\bar{u},\bar{t}-\Delta t) - f^\e(y,\bar{t}-\Delta t))- \tilde{P}(\bar{u})
\end{eqnarray*}
i.e.
\begin{eqnarray*}
0 
& \geq& f^\e(\bar{u},\bar{t})+ \max_{y \in \tilde{\Omega}}\left(  - f^\e(\bar{u},\bar{t}-\Delta t) + \Delta t  J_\e(\bar{u},y)(f^\e(\bar{u},\bar{t}-\Delta t) - f^\e(y,\bar{t}-\Delta t))\right)- \Delta t\tilde{P}(\bar{u}) \\ 
&=& f^\e(\bar{u},\bar{t}) + \max_{y \in \tilde{\Omega}}\left(-(1-\Delta t  J_\e(\bar{u},y))f^\e(\bar{u},\bar{t}-\Delta t) - \Delta t J_\e(\bar{u},y)f^\e(y,\bar{t}-\Delta t) \right) - \Delta t\tilde{P}(\bar{u}) \\
&\geq& f^\e(\bar{u},\bar{t})+ \max_{y \in \tilde{\Omega}}\bigg((1-\Delta t  J_\e(\bar{u},y))\big(g^\e(\bar{u},\bar{t})-g^\e(\bar{u},\bar{t}-\Delta t) + \eta \Delta t -f^\e(\bar{u},\bar{t})\big)\\
&+& \Delta t J_\e(\bar{u},y)\big(g^\e(\bar u,\bar{t})-g^\e(y,\bar{t}-\Delta t) + \eta \Delta t -f^\e(\bar{u},\bar{t})\big)\bigg)- \Delta t\tilde{P}(\bar{u}) \\
&=& \eta \Delta t + \Delta t \left( \frac{g^\e(\bar{u},\bar{t}) - g^\e(\bar{u},\bar{t}-\Delta t)}{\Delta t} + \max_{y \in \tilde{\Omega}}J_\e(\bar{u},y)  (g^\e(\bar{u},\bar{t}-\Delta t) - g^\e(y,\bar{t}-\Delta t))- \tilde{P}(\bar{u})  \right) \\
&\geq & \eta \Delta t + 0> 0,
\end{eqnarray*}
which is a contradiction.
\end{proof}

We now establish the existence and the regularity properties of a discrete solution.

\begin{lem}[Existence and Lipschitz regularity properties in time and space for the scheme \eqref{cauchy-J-discrete-fw}]\label{lem-existence-lip-fw}
Assume that assumptions~\ref{assum:Om}--\ref{assum:psi}, \ref{assum:gpos}--\ref{assum:gdec} and \ref{assum:psisssol-J}--\ref{assum:compatdomains} hold. Then there exists a discrete solution  $f^\e$ of \eqref{cauchy-J-discrete-fw} and for all $(u,v) \in \tilde{\Omega}^2$ and $t \in \set{t_1,\dots,t_{N_T}}$, the following holds
\begin{align}
\abs{f^\e(u,t) - f^\e(u,t-\Delta t)} &\leq L \Delta t, \label{lip-t-fw} \\
\abs{f^\e(u,t) - f^\e(v,t)} &\leq K\pa{|u-v|+\e} , \label{eq:globlip-space-J-fw}
\end{align}
where $L = \Lip{\tilde{\psi}} + \norm{\tilde{P}}_{L^\infty(\tilde{\Omega}\setminus\tilde{\Gamma})}$ and $K=2c_g^{-1}C_g\bpa{L + \norm{\tilde{P}}_{L^\infty(\tilde{\Omega}\setminus\tilde{\Gamma})}}m^{3/2}$.
\end{lem}

\begin{proof}
The existence of a solution is trivial. Moreover, since $\tilde \psi_b$ and $\bar f^\e$ are respectively sub- and super-solution of \eqref{cauchy-J-discrete-fw} and satisfy the boundary conditions, we get, by Lemma \ref{lem-comparison},  that $\tilde \psi_b\le f^\e\le\bar f^\e$.

For the Lipschitz regularity, we begin by showing that for any $0 < t \in \set{t_1,\dots,t_{N_T}}, u \in \tilde{\Omega}$,
\begin{eqnarray*}
|f^\e(u,t) - f^\e (u,0)| \leq L t.
\end{eqnarray*}
To do this, we  define
$$f_1^\e(u,t) = \tilde{\psi}(u) - Lt \qandq  f_2^\e(u,t) = \tilde{\psi}(u) + Lt.$$
In particular, $f_1^\e,f_2^\e$ are respectively sub- and super-solution of \eqref{cauchy-J-discrete-fw}. Indeed, on the one hand, we have
\begin{eqnarray}\label{sub-1}
&&\frac{f_1^\e(u,t) - f_1^\e(u,t - \Delta t)}{\Delta t} = \frac{\tilde{\psi}(u) - Lt- \tilde{\psi}(u) +L(t - \Delta t)}{\Delta t} = - L.
\end{eqnarray}
On the other hand, we have
\begin{equation}\label{sub-2}
\begin{aligned}
&- \max_{v\in\tilde{\Omega}} J_\e(u,v)(f_1^\e(u,t-\Delta t) - f_1^\e (v,t-\Delta t)) + \tilde{P}(u) \\
&= -\max_{v\in\tilde{\Omega}} \frac{g\pa{\frac{|u-v|}{\e}}}{\e C_g}  (\tilde{\psi}(u) - \tilde{\psi}(v)) + \tilde{P}(u) \\
&\geq -\max_{v\in\tilde{\Omega}} \frac{g\pa{\frac{|u-v|}{\e}}}{\e C_g} \Lip{\tilde{\psi}} |u-v|-\norm{\tilde{P}}_{L^\infty(\tilde{\Omega}\setminus\tilde{\Gamma})} \\
&= -L 
\end{aligned}
\end{equation}
Therefore, from \eqref{sub-1} and \eqref{sub-2} we get the conclusion. The proof for $f^\e_2$ is similar and we skip it.

Moreover, for any $(u,t) \in \tilde{\Gamma} \times  \set{t_1,\ldots,t_{N_T}} \cup \tilde{\Omega} \times \{ 0 \} $ we have
\begin{eqnarray}\label{compar-ineq}
f_1^\e(u,t) \leq f^\e (u,t) = \tilde{\psi}(u) \leq f_2^\e(u,t).
\end{eqnarray}
Hence, by the comparison principle in Lemma~\ref{lem-comparison}, we get that for any $u \in \tilde{ \Omega}, t\geq 0$,
\[
f^\e(u,0) - L t \leq f^\e(u,t) \leq  f^\e(u,0) + Lt.
\]
We now apply this estimate to get \eqref{lip-t-fw}. Let $ u \in \tilde{\Omega} \setminus \tilde{\Gamma}, t \in \{t_1,\dots t_{N_T}\}$ (the result being trivial if $u \in \tilde{\Gamma}$) and  set $s = t - \Delta t$. We have that $f^\e(u,s)$ is a solution of \eqref{cauchy-J-discrete-fw} with initial condition $f^\e(u,0)$ and $f^\e(u,s+\Delta t)$ is also a solution of \eqref{cauchy-J-discrete-fw} with initial condition $f^\e(u,\Delta t)$. Then by comparison principle Lemma~\ref{lem-comparison} and \eqref{compar-ineq}, we obtain for any $u \in \tilde{\Omega}$, $t>0$,
\begin{align*}
|f^\e(u,t) - f^\e(u,t-\Delta t)| 
&= |f^\e(u,s+ \Delta t) - f^\e(u,s)|\\
&\leq |f^\e(u,\Delta t) - f^\e(u,0)|\\
&\leq L \Delta t. 
\end{align*}
The proof of the space regularity estimate is the same as that of \eqref{eq:globlip-space-J}.
\end{proof}

\section{Well-posedness and regularity properties of \eqref{cauchy-J-discrete-bw}} 

For the reader's convenience, we establish the existence of a discrete solution for \eqref{cauchy-J-discrete-bw}.
\begin{lem}[Existence of discrete solution of \eqref{cauchy-J-discrete-bw}]\label{discrete-existence}
Assume that assumptions~\ref{assum:Om}--\ref{assum:psi}, \ref{assum:gpos}--\ref{assum:gdec} and \ref{assum:psisssol-J} hold. Then there exists a discrete solution  $f^\e$ of \eqref{cauchy-J-discrete-bw}. 
\end{lem}

\begin{proof}
The proof is very close to the one of Proposition~\ref{pro:existence-J} (and we recall all the notations there), and we therefore give here only a sketch of the proof. First, it is easy to check that $\tilde \psi_b$ and $\bar f^\e$ are respectively sub- and super-solution of \eqref{cauchy-J-discrete-bw} and satisfy the boundary conditions.

We assume that there exists a solution $f^{n}$ at step $n$ and we will construct a solution $f^{n+1}$ at step $n+1$.  Let us define
\[
f^{n+1}=\sup\set{w \textrm{ sub-solution at step } n+1 \textrm{ s.t } w\le \bar f^\e} .
\]
In particular, this set is nonempty since $\tilde \psi_b$ belongs to it. Moreover, we remark, by monotonicity, that if $\pa{f^{n+1,i}}_{i \in \N}$ is a family of discrete sub-solutions at step $n+1$, then $f^{n+1}=\sup_{i} f^{n+1,i}$ is still a sub-solution. Hence $f^{n+1}$ is a discrete sub-solution. Let us prove that $f^{n+1}$ is a super-solution. By contradiction, assume that there exists $\bar u\in   (\tilde \Omega \setminus \Gamma_n) \times \set{t_1,\ldots,t_{N_T}}$ such that (with the notation $f^n(u)=f(u,t_n)$)
\[
\frac{f^{n+1}(\bar u) - f^{n}(\bar u)}{\Delta t} < - \abs{\nabla_{J_\e}^- f^{n}(\bar u)}_\infty + \tilde{P}(\bar u).
\]
This implies in particular that $f^{n+1}(\bar u)<\bar f^\e(\bar u, t_{n+1})$. Now, let us consider the solution $w_{\bar u}$ of
\[
\frac{w_{\bar u} - f^{n}(\bar u)}{\Delta t} = - \max_{ v\in\tilde \Omega}J_\e (\bar u,v)(w_{\bar u}-f^{n+1}(v)) + \tilde{P}(\bar u).
\]
The existence of such a solution comes from the fact that the left hand-side is increasing in $w_{\bar u}$ while the right-hand side is non-increasing. Then, using the monotonicity of the scheme, it is easy to prove that $w_{\bar u}>f^{n+1}(\bar u)$ and $w$ defined by 
\[
w(u)=\left\{\begin{array}{ll}
w_{\bar u}&{\rm if}\; u=\bar u\\
f^{n+1}(u)& {\rm otherwise}
\end{array}
\right.
\]
is a discrete sub-solution of \eqref{cauchy-J-discrete-bw} at step $n+1$.  This contradicts the definition of $f^{n+1}$. The proof is completed.
\end{proof}

\section{Proof of Lemma~\ref{lem:compatassum-graph}}\label{sec:prooflemcompatassum-graph}
We will use again compactness of $\Omega$ and a covering argument with a finite $\delta$-net consisting of $N(\Omega,\delta)$ points, and conclude by the union bound, after using a standard estimate of $N(\Omega,\delta)$ (called the covering number of $\Omega$). We denote for short $[N]=\set{1,\ldots,N}$ for any $N \in \N^*$.

Let $S_\delta = \set{x_1,x_2,\ldots,x_{N(\Omega,\delta)}}$ be a $\delta$-net $\Omega$ in the Euclidian distance, i.e., $\Omega \subseteq \bigcup_{x \in S_\delta} \ball_\delta(x)$. We then have
\begin{align*}
\max_{x \in \Omega} d(x,V_n) \leq \max_{j \in [N(\Omega,\delta)]}\max_{x \in \ball_\delta(x_j)} d(x,V_n) .
\end{align*}
For each $j \in [N(\Omega,\delta)]$, let $Z_j$ be the number of random variables $\pa{u_i}_{i=1}^n$ falling into $\ball_{\delta}(x_j)$. Obviously, $Z_j$ is a Binomial random variable with parameters $(n,p_j)$, where $p_j = \mu(\ball_\delta(x_j)) \geq c \vol(\ball_\delta(0)) = c \delta^m \vol(\ball(0))$, where $c = \inf_{\Omega} \rho > 0$ by \ref{assum:densitylbd}, and we used the shorthand notation $\ball(0)$ for the unit Euclidian ball. Thus, using the union bound, we get
\begin{align*}
\Pr\pa{\max_{x \in \Omega} d(x,V_n) > 2\delta} 
&\leq \Pr\pa{\max_{j \in [N(\Omega,\delta)]}\max_{x \in \ball_\delta(x_j)} d(x,V_n) > 2\delta} \\
&\leq \sum_{j \in [N(\Omega,\delta)]}\Pr\pa{\max_{x \in \ball_\delta(x_j)} d(x,V_n) > 2\delta} \\
&\leq \sum_{j \in [N(\Omega,\delta)]}\Pr\pa{Z_j = 0} \\
&= \sum_{j \in [N(\Omega,\delta)]}(1-p_j)^n \\
&\leq N(\Omega,\delta) \pa{1-c \delta^m \vol(\ball(0))}^n .
\end{align*}
Since $\Omega$ is compact, there exists $r > 0$ such that $\Omega \subseteq r\ball(0)$. It then follows from standard estimates, see \cite[Lemma~4.10]{Pisier89} that
\[
N(\Omega,\delta) = N(\Omega/r,\delta/r) \leq \pa{1+\frac{2r}{\delta}}^m .
\]
We therefore arrive at the bound
\begin{align*}
\Pr\pa{\max_{x \in \Omega} d(x,V_n) > 2\delta} 
&\leq \pa{1+\frac{2r}{\delta}}^m \pa{1-c \delta^m \vol(\ball(0))}^n \\
&\leq e^{-nc\delta^m\vol(\ball(0))+m\log\pa{1+\frac{2r}{\delta}}} .
\end{align*}
Take $\delta^m = \frac{(1+\tau)}{c\vol(\ball(0) }\frac{\log n}{n}$, for any $\tau > 0$. Thus, for $n$ large enough, one has $\delta \leq r$, and in turn the above bound becomes
\begin{align*}
\Pr\pa{\max_{x \in \Omega} d(x,V_n) > 2\delta} 
&\leq c (3r)^m \vol(\ball(0)) e^{-(1+\tau)\log n - \log (1+\tau) - \log\log n + \log n} \\
&\leq c (3r)^m \vol(\ball(0)) e^{-\tau\log n} = c (3r)^m \vol(\ball(0)) n^{-\tau} .
\end{align*}
By the Stirling formula, we have
\[
\vol(\ball(0))=\frac{2\pi^{m/2}}{m\Gamma(m/2)} = \frac{1}{\sqrt{m\pi}} \pa{\frac{2\pi e}{m}}^{m/2} e^{\theta(m/2)/(6m)}
\] 
with $\theta(m/2) \in [0,1]$. Thus, taking 

\[
\e_{n}^{1+\nu} = 8a^{-1}\sqrt{2\pi e^{4/3}}\pa{\frac{(1+\tau)}{\sqrt{\pi m}c}}^{1/m} \pa{\frac{\log n}{n}}^{1/m}, 
\]
we have $a\e_{n}^{1+\nu}/(4\sqrt{m}) \geq \delta$, and thus \ref{assum:compatdomains} holds with probability at least $1-K_2n^{-\tau}$.

Let us turn to the estimating the probability of the event

\[
\set{\distH(\Gamma,\Gamma_n) \leq a\e_n^{1+\nu}/(2\sqrt{m})} .
\] 

First, with the construction of Definition~\ref{def:randgraph}, one can assert that $\Gamma_n \neq \emptyset$ with probability larger than $1 - c (3r)^m \vol(\ball(0)) n^{-\tau}$. To show this, we argue by contradiction, assuming that $\forall u \in V_n$, $d(u,\Gamma) > 2\delta$, which entails that
\[
|u-x| > 2\delta, \quad \forall (u,x) \in V_n \times \Gamma .
\]
Let $j \in [N(\Omega,\delta)]$ such that $\Gamma \cap \ball_\delta(x_j) \neq \emptyset$ (which exists by definition of the $\delta$-net). We have shown above that with probability at least $1 - c (3r)^m \vol(\ball(0)) n^{-\tau}$, each ball $\ball_\delta(x_j)$ contains at least one point $u \in V_n$, and thus, for such a point, $|u-x| \leq 2\delta$ for all $x \in \Gamma \cap \ball_\delta(x_j)$, leading to a contradiction. In turn, we deduce that with the same probability, we have
\[
\max_{u \in \Gamma_n} d(x,\Gamma) \leq 2\delta \leq a\e_n^{1+\nu}/(2\sqrt{m}) .
\]

To conclude, it remains to show that
\[
\max_{x \in \Gamma} d(x,\Gamma_n) \leq 2\delta ,
\]
with the same probability. For this, let $\set{x_j \in S_\delta:~ \Gamma \cap \ball_\delta(x_j) \neq \emptyset}$. This is a subcover of $\Omega$ which is a $\delta$-net of $\Gamma$. Thus Arguing as we did above to bound $\max_{x \in \Omega} d(x,V_n)$, we get the claimed bound. Finally, the bound on $\distH(\Gamma,\Gamma_n)$ follows from above using the union bound. The latter also yields that the event $\En$ holds with the given probability. \qed

\medskip

\paragraph{\textbf{Acknowledgement.}}
This work was supported by the Normandy Region grant MoNomads, and partly by the European Union's Horizon 2020 research and innovation programme under the Marie Sk{\l}odowska-Curie grant agreement No 777826 (NoMADS). We would like also to thank A. El Moataz for fruitful discussions.

\bibliographystyle{plain}
\bibliography{biblio}

\begin{thebibliography}{10}

\bibitem{Achdou13}
Y.~Achdou, F.~Camilli, A.~Cutri, and N.~Tchou.
\newblock {H}amilton--{J}acobi equations constrained on networks.
\newblock {\em Nonlinear Differential Equations and Applications},
  20:413--445., 2013.

\bibitem{Ambrosio96}
L.~Ambrosio and H.~Mete Soner.
\newblock {Level set approach to mean curvature flow in arbitrary codimension}.
\newblock {\em Journal of Differential Geometry}, 43(4):693 -- 737, 1996.

\bibitem{BarlesBook}
G.~Barles.
\newblock {\em Solutions de viscosit\'e des \'equations de
  {H}amilton-{J}acobi}.
\newblock Math\'ematiques et Applications. Springer- Verlag, 1994.

\bibitem{barles11}
G.~Barles.
\newblock First-order {H}amilton-{J}acobi equations and applications, 2011.
\newblock Lecture Notes, CIME course.

\bibitem{Barth98}
T.~J. Barth and J.~A. Sethian.
\newblock Numerical schemes for the {H}amilton--{J}acobi and level set
  equations on triangulated domains.
\newblock {\em Journal of Computational Physics}, 145(1):1--40, 1998.

\bibitem{Berkolaiko13}
G.~Berkolaiko and P.~Kuchment.
\newblock {\em Introduction to Quantum Graphs}, volume 186 of {\em Mathematical
  Surveys and Monographs}.
\newblock AMS, 2013.

\bibitem{Bollobas07}
B.~Bollob\'{a}s, S.~Janson, and O.~Riordan.
\newblock The phase transition in inhomogeneous random graphs.
\newblock {\em Random Structures \& Algorithms}, 31(1):3--122, 2007.

\bibitem{Calder19}
J.~Calder.
\newblock Consistency of {L}ipschitz learning with infinite unlabeled data and
  finite labeled data.
\newblock {\em SIAM Journal on Mathematics of Data Science}, 1:780--812, 2019.

\bibitem{Camilli11}
F.~Camilli, A.~Festa, and D.~Schieborn.
\newblock Shortest paths and {E}ikonal equations on a graph.
\newblock arXiv:1105.5725, 2011.

\bibitem{Camilli13}
F.~Camilli, A.~Festa, and D.~Schieborn.
\newblock An approximation scheme for a {H}amilton-{J}acobi equation defined on
  a network.
\newblock {\em Applied Numerical Mathematics}, 73:33 -- 47, 2013.

\bibitem{Canino88}
A.~Canino.
\newblock On $p$-convex sets and geodesics.
\newblock {\em Journal of Differential Equations}, 75(1):118--157, 1988.

\bibitem{Carlini13}
E.~Carlini, M.~Falcone, and P.~Hoch.
\newblock A generalized fast marching method on unstructured triangular meshes.
\newblock {\em SIAM Journal on Numerical Analysis}, 51(6):2999--3035, 2013.

\bibitem{Caroccia20}
M.~Caroccia, A.~Chambolle, and D.~Slep{\v{c}}ev.
\newblock {Mumford-Shah} functionals on graphs and their asymptotics.
\newblock {\em Nonlinearity}, 33(8):3846--3888, jun 2020.

\bibitem{Thibault10}
G.~Colombo and L.~Thibault.
\newblock Prox-regular sets and applications.
\newblock In D.~Gao and D.~Motreanu, editors, {\em Handbook of Nonconvex
  Analysis and Applications}, pages 99--182. International Press, Sommerville,
  2010.

\bibitem{cil92}
M.~G. Crandall, H.~Ishii, and P.-L. Lions.
\newblock User's guide to viscosity solutions of second order partial
  differential equations.
\newblock {\em Bulletin of the American Mathematical Society}, 27(1):1--67,
  1992.

\bibitem{cl83}
M.~G. Crandall and {P.-L.} Lions.
\newblock Viscosity solutions of {H}amilton-{J}acobi equations.
\newblock {\em Transactions of the American Mathematical Society},
  277(1):1--42, 1983.

\bibitem{cl84}
M.~G. Crandall and P.-L. Lions.
\newblock Two approximations of solutions of {H}amilton-{J}acobi equations.
\newblock {\em Mathematics of Computation}, 43(167):1--19, 1984.

\bibitem{cl86}
M.~G. Crandall and {P.-L.} Lions.
\newblock On existence and uniqueness of solutions of {H}amilton-{J}acobi
  equations.
\newblock {\em Nonlinear Analysis: Theory, Methods \& Applications},
  10(4):353--370, 1986.

\bibitem{de2013courant}
C.~A. De~Moura and C.~S. Kubrusly, editors.
\newblock {\em The {C}ourant--{F}riedrichs--{L}ewy ({CFL}) condition: 80 Years
  After Its Discovery}.
\newblock Birkh{\"a}user, 2013.

\bibitem{Desquesnes17}
X.~Desquesnes and A.~Elmoataz.
\newblock Non-monotonic front propagation on weighted graphs with applications
  in image processing and high-dimensional data classification.
\newblock {\em IEEE Journal of Selected Topics in Signal Processing},
  11(6):897--907, 2017.

\bibitem{del13}
X.~Desquesnes, A.~Elmoataz, and O.~L\'{e}zoray.
\newblock {E}ikonal equation adaptation on weighted graphs: fast geometric
  diffusion process for local and non-local image and data processing.
\newblock {\em Journal of Mathematical Imaging and Vision}, 46(2):238--257,
  2013.

\bibitem{elb08}
A.~Elmoataz, O.~Lezoray, and S.~Bougleux.
\newblock Nonlocal discrete regularization on weighted graphs: a framework for
  image and manifold processing.
\newblock {\em IEEE Transactions on Image Processing}, 17(7):1047--1060, 2008.

\bibitem{nf08}
N.~Forcadel.
\newblock An error estimate for a new scheme for mean curvature motion.
\newblock {\em SIAM Journal on Numerical Analysis}, 46(5):2715--2741, 2008.

\bibitem{Garavello06}
M.~Garavello and B.~Piccoli.
\newblock {\em Traffic flow on networks}.
\newblock AIMS Series on Applied Mathematics. Springfield, 2006.

\bibitem{Slepcev2016}
N.~Garc{\'\i}a~Trillos and D.~Slep{\v{c}}ev.
\newblock Continuum limit of total variation on point clouds.
\newblock {\em Archive for Rational Mechanics and Analysis}, 220(1):193--241,
  April 2016.

\bibitem{Trillos16}
N.~Garc{\'\i}a~Trillos, D.~Slep{\v c}ev, J.~Von Brecht, T.~Laurent, and
  X.~Bresson.
\newblock Consistency of {C}heeger and ratio graph cuts.
\newblock {\em Journal of Machine Learning Research}, 6268--6313:1532--4435,
  2016.

\bibitem{Hafiene20}
Y.~Hafiene, M.J. Fadili, C.~Chesneau, and A.~Elmoataz.
\newblock Continuum limit of the nonlocal $p$-{L}aplacian evolution problem on
  random inhomogeneous graphs.
\newblock {\em ESAIM: Mathematical Modelling and Numerical Analysis (M2AN)},
  54(2):565--589, 2020.

\bibitem{Hafiene18}
Yosra Hafiene, Jalal Fadili, and A.~Elmoataz.
\newblock Nonlocal {$p$}-{L}aplacian evolution problems on graphs.
\newblock {\em SIAM Journal on Numerical Analysis}, 56(2):1064--1090, 2018.

\bibitem{Imbert13}
C.~Imbert, R.~Monneau, and H.~Zidani.
\newblock A {H}amilton-{J}acobi approach to junction problems and application
  to traffic flows.
\newblock {\em ESAIM Control Optim. Calc. Var.}, 19:129--166, 2013.

\bibitem{i86}
H.~Ishii.
\newblock Existence and uniqueness of solutions of {H}amilton-{J}acobi
  equations.
\newblock {\em Funkcialaj Ekvacioj. Serio Internacia}, 29(2):167--188, 1986.

\bibitem{Ishii87}
H.~Ishii.
\newblock Perron's method for {H}amilton-{J}acobi equations.
\newblock {\em Duke Mathematical Journal}, 55(2):369--384, 1987.

\bibitem{Oberman15}
Manfredi Juan~J., Oberman Adam~M., and Sviridov Alexander~P.
\newblock Nonlinear elliptic partial differential equations and $p$-harmonic
  functions on graphs.
\newblock {\em Differential Integral Equations}, 28(1/2):79--102, 01 2015.

\bibitem{kimmel98}
R.~Kimmel and J.~Sethian.
\newblock Computing geodesic paths on manifolds.
\newblock {\em Proc. Natl. Acad. Sci.}, pages 8431--8435, 1998.

\bibitem{medv}
G.~S. Medvedev.
\newblock The nonlinear heat equation on dense graphs.
\newblock {\em SIAM Journal on Mathematical Analysis}, 46(4):2743--2766, 2014.

\bibitem{medvedevrandom}
G.~S. Medvedev.
\newblock The nonlinear heat equation on {$W$}-random graphs.
\newblock {\em Archive for Rational Mechanics and Analysis}, 212(3):781--803,
  2014.

\bibitem{Memoli05}
F.~M\'emoli and G.~Sapiro.
\newblock Distance functions and geodesics on submanifolds of {$\mathbb{R}^d$}
  and point clouds.
\newblock {\em SIAM Journal of Applied Mathematics}, 65:1227--1260, 01 2005.

\bibitem{Sethian88}
S.~Osher and J.~Sethian.
\newblock Fronts propagating with curvature-dependent speed-algorithms based on
  hamilton-jacobi formulations.
\newblock {\em J. Comput. Phys.}, 79(1):12--49, 1988.

\bibitem{rossi}
M.~P{\'e}rez-LLanos and J.~D. Rossi.
\newblock Numerical approximations for a nonlocal evolution equation.
\newblock {\em SIAM Journal on Numerical Analysis}, 49(5):2103--2123, 2011.

\bibitem{Pisier89}
G.~Pisier.
\newblock {\em The volume of convex bodies and Banach space geometry}.
\newblock Cambridge University Press, 1989.

\bibitem{Pokornyi04}
Y.~V. Pokornyi and A.~V. Borovskikh.
\newblock Differential equations on networks (geometric graphs).
\newblock {\em J. Math. Sci.}, 119:691--718, 2004.

\bibitem{Poly84}
J.B. Poly.
\newblock Fonction distance et sigularit{\'e}s.
\newblock {\em Bulletin des Sciences Math{\'e}matiques}, 108(2):187--195, 1984.

\bibitem{Roith21}
Tim Roith and Leon Bungert.
\newblock Continuum limit of {L}ipschitz learning on graphs.
\newblock arXiv 2012.03772, 2021.

\bibitem{Thibault19}
D.~Salas and L.~Thibault.
\newblock On characterizations of submanifolds via smoothness of the distance
  function in {H}ilbert spaces.
\newblock {\em Journal of Optimization Theory and Applications},
  182(1):189--210, 2019.

\bibitem{Schieborn13}
D.~Schieborn and F.~Camilli.
\newblock Viscosity solutions of {E}ikonal equations on topological networks.
\newblock {\em Calc. Var.}, 46:671--686, 2013.

\bibitem{Sethian96}
J.~Sethian.
\newblock A fast marching level-set method for monotonically advancing fronts.
\newblock {\em Proc. Natl. Acad. Sci.}, 93:1591--1595, 1996.

\bibitem{Sethian99}
J.~Sethian.
\newblock {\em Level Set Methods and Fast Marching Methods: Evolving Interfaces
  in Computational Geometry, Fluid Mechanics, Computer Vision and Materials
  Science}.
\newblock Cambridge University Press, 1999.

\bibitem{Shapiro94}
A.~Shapiro.
\newblock Existence and differentiability of metric projections in {H}ilbert
  spaces.
\newblock {\em SIAM Journal on Optimization}, 4(1):130--141, 1994.

\bibitem{Shu18}
Y.~Shu.
\newblock {H}amilton-{J}acobi equations on graph and applications.
\newblock {\em Potential Analysis}, 48(2):125--157, 2018.

\bibitem{Slepcev2017}
D.~Slep{\v{c}}ev and M.~Thorpe.
\newblock Analysis of $p$-{L}aplacian regularization in semi-supervised
  learning.
\newblock {\em SIAM Journal on Mathematical Analysis}, 51(3):2085--2120, 2019.

\bibitem{Ta09}
V.-T. Ta, A.~Elmoataz, and O.~L\'ezoray.
\newblock Adaptation of eikonal equation over weighted graphs.
\newblock In {\em International Conference on Scale Space and Variational
  Methods in Computer Vision}, volume 5567 of {\em LNCS}, pages 187--199.
  Springer, 2009.

\bibitem{Ta11}
V.T. Ta, A.~Elmoataz, and O~L\'ezoray.
\newblock Nonlocal {PDEs}-based morphology on weighted graphs for image and
  data processing.
\newblock {\em IEEE Transactions on Image Processing}, 20(6):1504--1516, 2011.

\bibitem{Toutain16}
M.~Toutain, A.~Elmoataz, F.~Lozes, and A.~Mansouri.
\newblock Non-local discrete $\infty$-{P}oisson and {H}amilton-{J}acobi
  equations: From stochastic game to generalized distances on images, meshes,
  and point clouds.
\newblock {\em Journal of Mathematical Imaging and Vision}, 55(2):229--241,
  2016.

\end{thebibliography}

\end{document}